\documentclass[a4paper]{amsart}

\usepackage{amsmath}
\usepackage{amsfonts}
\usepackage{amsthm}
\usepackage{graphicx}
\usepackage{eucal}
\usepackage{amscd}
\usepackage[all,2cell]{xy}
\usepackage{amssymb}
\usepackage{mathrsfs}

\usepackage{tikz-cd}
\usetikzlibrary{matrix,arrows}

\newcommand{\Cal}[1]{{\mathcal #1}}
\newcommand{\End}{\operatorname{End}}

\newcommand{\Sym}{\operatorname{Sym}}

\newcommand{\Aut}{\operatorname{Aut}}

\newcommand{\Der}{\operatorname{Der}}

\newcommand{\Gp}{\mathsf{Gp}}
\newcommand{\Hp}{\mathsf{Heap}}
\def\Z{ \mathbb{Z} }
\DeclareMathOperator{\op}{op}

\newtheorem{thm}{Theorem}[section]

\newtheorem{cor}[thm]{Corollary}
\newtheorem{Lemma}[thm]{Lemma}
\newtheorem{prop}[thm]{Proposition}
\newtheorem{rem}[thm]{Remark}
\newtheorem{defin}[thm]{Definition}

\theoremstyle{definition}

\newtheorem{exam}[thm]{Example}
\newtheorem{exams}[thm]{Examples}

\theoremstyle{remark}

\numberwithin{equation}{section}

\newenvironment{Proof}{{\sc Proof.}\ }{~\rule{1ex}{1ex}\vspace{0.2truecm}}
\def\Z{ \mathbb{Z} }
\def\R{ \mathbb{R} }
\def\N{ \mathbb{N} }

\begin{document}

\title{Heaps and trusses\\ }

 \author{Mar\'{\i}a Jos\'e Arroyo Paniagua}
 \address{Departamento de Matem\'aticas. Divisi\'on de Ciencias B\'asicas e Ingenier\'ia. Universidad Aut\'onoma Metropolitana, Unidad Iztapalapa,  CP 09310, Ciudad de M\'exico, M\'exico.}
 \email{mja@xanum.uam.mx}
\author{Alberto Facchini}
	\address{Dipartimento di Matematica ``Tullio Levi-Civita'', Universit\`a di 
Padova, 35121 Padova, Italy}
 \email{facchini@math.unipd.it}
\thanks{The second author was partially supported by Ministero dell'Universit\`a e della Ricerca (Progetto di ricerca di rilevante interesse nazionale ``Categories, Algebras: Ring-Theoretical and Homological Approaches (CARTHA)''), Fondazione Cariverona (Research project ``Reducing complexity in algebra, logic, combinatorics - REDCOM'' within the framework of the programme Ricerca Scientifica di Eccellenza 2018), and the Department of Mathematics ``Tullio Levi-Civita'' of the University of Padua (Research programme DOR1828909 ``Anelli e categorie di moduli'').}

\begin{abstract}
We study commutators of congruences, idempotent endomorphisms and semidirect-product decompositions of heaps and trusses.\end{abstract}

\maketitle

\section{Introduction}

Heaps were already considered by Pr\"ufer \cite{10} and Baer \cite{1} more than a century ago, but they have received little attention for these 100 years (see, for instance, \cite{6}). Heaps are an algebraic structure endowed with a {\em ternary} operation, and perhaps that is why they have not generated much interest. Recently though, their importance for trusses, a notion due to T.~Brzezi{\'n}ski, have been motivating because of the relations of these algebraic structures with set-theoretic solutions of the Yang-Baxter equation and with left skew braces \cite{2022, 2018, Brebetween, Bre, Brzez}. 

There are three main reasons as to why we study heaps and trusses:

\noindent (1) They form the most natural examples of varieties with a Mal'tsev term (the operation itself!).

\noindent (2) Their natural relation with the notion of connector \cite{BG}, hence with the notion of commutator of congruences. 

\noindent (3) They give the most immediate description of our Newtonian Universe, that is, of our $3$-dimensional Euclidean geometrical affine real space (see Example~\ref{111}(a)). Instead of giving the standard model of the Newtonian Universe as the $3$-dimensional vector space $\R^3$, which presume that we fix a point (the origin), for heaps and trusses a privileged point is not fixed. It is nice to recall here that it was already observed in the fifteenth century by Nicholas of Cusa that ``the Universe itself is like an infinite circle, which has its center everywhere'', and therefore that, like in modern Cosmology, all observers are formally equivalent. Hence there is no doubt of how natural algebraic structures heaps and trusses are.

In this paper,  beyond the basics on heaps and trusses, we focus mainly on their commutators and semidirect products. We study the notion of commutator for heaps and trusses on the one hand because the ternary operation yields a particularly natural notion of connector, as we have already remarked in (2) above. Thus is natural that the question ``Huq=Smith?'' must be revisited for the algebraic structures with a ternary operation we are studying.
On the other hand, a computation of commutators of congruences in any algebraic structure leads immediately to the notions of abelian structure, solvability and nilpotency. Finally, we focus on the study of idempotent endomorphisms of heaps and trusses because idempotent endomorphisms of any algebraic structure immediately lead to the study of semidirect decompositions of the algebraic structure itself. 
Derivations of trusses are also briefly discussed.

We have tried to make this paper as self contained as possible. The second author is very grateful to Professor George Janelidze for an interesting discussion on the role of heaps and trusses.

\section{Basic notions and notation}

\subsection{Ternary operations, Mal'tsev operations.} In this paper, we will consider sets $X$ endowed with a ternary operation $p\colon X\times X\times X\to X$ (no identity is required to be satisfied, at the moment). These pairs  $(X,p)$ form a variety of algebras in the sense of Universal Algebra. Their morphisms $f\colon (X,p)\to (X',p')$ are the mappings $f\colon X\to X'$ such that $p'(f(x),f(y),f(z))=f(p(x,y,z))$ for every $x,y,z\in X$. In particular, these algebras $(X,p)$ are the objects of a category, whose initial object is the empty set $\emptyset$ (with its unique ternary operation), and whose terminal objects are the singletons (with their unique ternary operation). We will denote by $\ast$ any such algebra with one element.

Let $p\colon X\times X\times X\to X$ be a ternary operation on the set $X$. We say that $p$ is a {\em Mal'tsev operation} if $p(x,x,y)=y$ and $p(x,y,y)=x$ for every $x,y\in X$. It will be often convenient to replace the ternary operation $p$ on a set $X$ with an indexed family $\{\,b_y\mid y\in X\,\}$ of binary operations $b_y\colon X \times X\to X$ defined by $b_y(x,z)=p(x,y,z)$ for every $x,y,z\in X$. The family of binary operations $b_y$ is indexed in $X$ itself. Correspondingly, we get a family of {\em magmas} (=\,sets with a binary operation) $(X,b_y)$, which is again indexed in $X$ itself. We will also often use the notations $[x,y,z]$ instead of $p(x,y,z)$, and $x\cdot_yz$ instead of $b_y(x,z)$.

\begin{Lemma} A ternary operation $p$ on a set $X$ is a Mal'tsev operation if and only if, for the corresponding indexed family $\{\,b_y\mid y\in X\,\}$ of binary operations, the element $y$ is a two-sided identity of the magma $(X,b_y)$ for every $y\in X$.\end{Lemma}

Notice that in a magma, that is, a set with a not-necessarily associative operation, a two-sided identity, when it exists, is unique. 

\medskip

A ternary operation $p$ on a set $X$ is {\em commutative} if $p(x,y,z)=p(z,y,x)$ for every $x,y,z\in~X$. Hence a  ternary operation $p$ on $X$ is commutative if and only if all binary operations $b_y$, $y\in X$, are commutative.

\subsection{Associative ternary operations}

A ternary operation $p$ on a set $X$ is {\em associative} if $p(p(x,y,z),w,u)=p(x,y,p(z,w,u))$ for every $x,y,z,w,u\in X$.

\medskip

For any magma $M$, there is a natural left action (the ``canonical Cayley left representation'') $\lambda_M\colon M\to M^M$, that maps any element $x\in M$ to left multiplication $\lambda_M(x)=\lambda_M^x$ by $x$, where $\lambda_M^x\colon M\to M$ is defined by  $\lambda_M^x(y)=xy$ for every $y\in X$. The mapping $\lambda_M$ is a magma morphism into the monoid $M^M$ of all mappings $M\to M$ if and only if $M$ is a semigroup. Similarly on the right: there is a natural right action $\rho_M\colon M\to M^M$, that maps any element $x\in M$ to right multiplication $\rho_M(x)=\rho_M^x$ by $x$, where $\rho_M^x\colon M\to M$ is defined by $\rho_M^x(y)=yx$ for every $y\in X$. The mapping $\rho_M$ is a magma antihomomorphism if and only if $M$ is a semigroup. Thus, for every magma $M$, the set $M$ is both a left $M$-set and a right $M$-set in a natural way.

Now that for a ternary operation $p$ on a set $X$ we have all the magmas $(X,b_x)$, i.e., all $(X,b_x)$-set structures $\lambda_{(X,b_x)}\colon(X,b_x)\to X^X$, $\lambda_{(X,b_x)}\colon u\in X\mapsto \lambda_{(X,b_x)}^u\in X^X$, where $\lambda_{(X,b_x)}^u\colon z\mapsto b_x(u,z)$, we want all these $(X,b_x)$-set structures on the set $X$ to be pairwise compatible. That is, for every $x,y\in X$, the set $X$ is both a left $(X, b_x)$-set and a right $(X, b_y)$-set, and we want any two of these structures to be compatible. Like in the case of an $R$-$S$-bimodule over two rings $R$ and $S$, this amounts to requiring that $(u\cdot_xz)\cdot_yw=u\cdot_x(z\cdot_yw)$ for every $z,u,w\in X$. Equivalently, to requiring that $\rho_{(X,b_y)}^w\circ\lambda_{(X,b_x)}^u=\lambda_{(X,b_x)}^u\circ \rho_{(X,b_y)}^w$ for every $u,w\in X$. Hence we say that all the magma structures $(X,b_x)$ are {\em compatible} if $(u\cdot_xz)\cdot_yw=u\cdot_x(z\cdot_yw)$ for every $x,y,z,u,w\in X$, or, equivalently, if $\rho_{(X,b_y)}^w\circ\lambda_{(X,b_x)}^u=\lambda_{(X,b_x)}^u\circ \rho_{(X,b_y)}^w$ for every $x,y,u,w\in X$. Clearly:

\begin{Lemma} A ternary operation $p$ on a set $X$ is associative if and only if all the corresponding magma structures $(X,b_x)$, $x\in X$, are pairwise compatible.\end{Lemma}

In particular, for $x=y$, we get that if a ternary operation $p$ on a set $X$ is associative, then all the binary operations $b_x$ are associative, i.e., that all the magmas 
$(X,b_x)$, $x\in X$, are semigroups.

\medskip

As we have said above, for any semigroup $S$, there is a semigroup morphism $\lambda\colon S\to S^S$ of $S$ into the monoid $S^S$ of all mappings $S\to S$ with composition of mappings. It associates with any element $x\in S$ left multiplication $\lambda_S^x\colon S\to S$ by $x$. The image $\lambda(S)$ is a subsemigroup of $S^S$. The kernel of the representation $\lambda$ is the congruence $\sim$ on the semigroup $S$ defined, for every $x,y\in S$, by $x\sim y$ if $xz=yz$ for every $z\in S$. Let us see how this elementary fact generalizes and applies to sets $X$ with an associative ternary operation~$p$:

\begin{prop}\label{2.3} Let $p$ be an associative ternary operation on a set $X$. For every $x,y\in X$, let $\tau_x^y\colon X\to X$ be the mapping defined by $\tau_x^y\colon z\in X\mapsto p(x,y,z)$ for every $z\in X$. Then:

{\rm (a)} $S_X:=\{\,\tau_x^y\mid x,y\in X\,\}$ is a subsemigroup of the monoid $X^X$.

{\rm (b)} 
%
The image of the canonical mapping $\lambda_{(X,b_y)}\colon (X,b_y)\to (X^X,\circ)$ is contained in $S_X$ for each $y\in X$.

{\rm (c)} $S_X=\bigcup_{y\in X}\lambda_{(X,b_y)}(X,b_y)$.\end{prop}

\begin{Proof} (a) $\tau_x^y\circ\tau_u^v(z)=p(x,y,p(u,v,z))=p(p(x,y,u),v,z)=\tau_{p(x,y,u)}^v(z)$ for every $z\in X$, so that $\tau_x^y\circ\tau_u^v=\tau_{p(x,y,u)}^v$. In particular, $S_X$ is a multiplicatively closed subset of $X^X$.

(b) 
The image of $\lambda_{(X,b_y)}\colon (X,b_y)\to (X^X,\circ)$ is $$\{\,\lambda_{(X,b_y)}^x\mid x\in X\,\}=\{\,p(x,y,-)\mid x\in X\,\}=\{\,\tau_x^y\mid x\in X\,\}\subseteq S_X.$$

(c) follows immediately from (b).
\end{Proof}

The mapping $\tau_x^y\colon X\to X$ defined in the statement of Proposition~\ref{2.3} is often called the {\em translation} from $x$ to $y$. A set $X$ with an associative ternary operation is a {\em semiheap}.

\section{Heaps}
A set $X$
 with an associative Mal'tsev operation $[-,-,-]$ is called a {\em heap}. A mapping $f\colon (X,[-,-,-])\to (X',[-,-,-])$ between two heaps is a {\em heap morphism} if $$f([x,x',x''])=[f(x), f(x'), f(x'')]$$ for every $x,x',x''\in X$. The category of heaps will be denoted by $\Hp$. In $\Hp$, the initial object is $\emptyset$ and the terminal object is $\ast$.
 
 Let's see how Proposition \ref{2.3} applies to non-empty heaps:
 
 \begin{thm}\label{X} Let $(X,p)$ be a non-empty heap.
Then the subsemigroup $S_X:=\{\,\tau_x^y\mid x,y\in X\,\}$ of the monoid $X^X$ is a subgroup of the symmetric group $\Sym_X$ on the set $X$, and all the monoids $(X,b_x)$, $x\in X$, are groups isomorphic to the group $S_X$ via the canonical Cayley left representations $\lambda_{(X,b_y)}\colon (X,b_y)\stackrel{\cong}{\longrightarrow}(S_X,\circ)$.\end{thm}

\begin{Proof} The monoids $(X,b_y)$ are groups because, for every $x\in X$, the element $p(x,y,x)$ is the inverse of $y$ in $(X,b_x)$. The monoid $S_X$ is a group because $\tau_y^x$ is the inverse of any $\tau_x^y$ in $S_X$.  The group morphisms $\lambda_{(X,b_y)}\colon (X,b_y)\to (S_X,\circ)$ are injective mappings because the monoids $(X,b_y)$ have an identity. It remains to show that every element of $S_X$ is in the image of $\lambda_{(X,b_y)}$. Now an arbitrary element of $S_X$ is of the form $\tau_u^v$. In order to prove that such an element is in the image of $\lambda_{(X,b_y)}$, it suffices to prove that $\tau_u^v=\tau_y^{p(v,u,y)}$. Passing to the inverse mappings, we must prove equivalently that $\tau_v^u=\tau_{p(v,u,y)}^y$, i.e., that $\tau_v^u(t)=\tau_{p(v,u,y)}^y(t)$ for every $t\in X$. This is easy.\end{Proof}

A subset $S$ of a heap is a {\em subheap} if $[x, y, z] \in S$ for every $x, y, z\in S$.

\begin{exams}\label{111} {\rm (a) The most natural example of heap is the following. In the ``$3$-dimensional Euclidean geometrical real space'' of Newtonian Physics the set $E_3$ of all points doesn't have a natural group structure (or a vector-space structure): the sum of two points doesn't have a natural meaning. But as soon as we fix a point (an {\em origin}), we can define an addition using the Parallelogram Rule, and we get an abelian group. In fact, we get a $3$-dimensional vector space over the field of real numbers. Hence $E_3$ does not have a natural group structure, if we want it we need the unnatural choice of an origin. But $E_3$ does have a natural heap structure: if $A,B,C\in E_3$, we can define $p(A,B,C)$ with the Parallelogram Rule, so that $A,B,C,p(A,B,C)$ are, orderly, the vertex of a parallelogram, and in this way we get a heap $(E_3,p)$.

(b) We can fix any line in the space, getting a subheap $E_1$ of the previous example $E_3$ in~(a).

(c) Fix any group $G$ and define a ternary operation $p$ on $G$ setting $p(x,y,z)=xy^{-1}z$ for every $x,y,z\in G$. Then $(G,p)$ is a heap. It can be proved \cite[Lemma~2.1(3)]{Bre} that every non-empty heap is of this form, and there is a natural functor of the category of groups into the categories of heaps. Nevertheless these two categories are not equivalent, for instance the category of heaps does not have a null object (the category of groups and the category of heaps are not equivalent categories also if we eliminate the empty heap from the objects of the category of heaps).

(d) As an example of a semiheap that is not a heap, consider any lattice $(L,\vee,\wedge)$ and define $p$ via $p(x,y,z)=x\vee y\vee z$ for every $x,y,z\in L$. Then $(L,p)$ is a semiheap that is not a heap, except for the trivial case of $L$ a lattice with one element. For instance, let $L$ be the lattice $\R$ of real numbers with their natural order. Then Theorem~\ref{X} does not hold for the semiheap $(\R,p)$. For instance, the images of all canonical Cayley left representations $\lambda_{(\R,b_y)}$ are properly contained in the semigroup $S_\R$, for every $y\in  \R$.}\end{exams}

\begin{Lemma}\label{3.3} {\rm \cite[Lemma 2.6]{Bre}} The following conditions are equivalent for a non-empty subheap $S$ of a heap $X$:

{\rm (a)} there exists $e \in S$ such that for every $x \in X$ and every $s\in S$ there exists $t\in S$ such that
$$[x, e,s] = [t, e,x].$$ 

{\rm (b)} For every $x \in X$ and every $e,s\in S$ there exists $t\in S$ such that
$[x, e,s] = [t, e,x]$. 

{\rm (c)} $[[x, e, s], x, e]\in S$ for every $x\in X$ and every $e,s\in S$.
\end{Lemma}
     
 A subheap $S$ of a heap $X$ is said to be a {\em normal subheap} if it is non-empty and satisfies the equivalent conditions of Lemma~\ref{3.3}.

 \begin{cor}\label{3.4} {\rm \cite[Corollary 2.7]{Bre}} The following conditions are equivalent for a subset $S$ of a heap $H$:

{\rm (a)} there exists $e \in S$ such that $S$ is a normal subgroup of $(X,b_e)$.

{\rm (b)} $S$ is non-empty and $S$ is a normal subgroup of $(X,b_e)$ for every $e \in S$.

{\rm (c)} $S$ is a normal subheap of $X$.\end{cor}

\begin{rem} {\rm Some care is necessary here.
   We have chosen our terminology in such a way that the empty set is a heap, the empty subset is a subheap of every heap, but normal subheaps are non-empty by definition. As a consequence:

   \noindent (a) Subheaps of a heap form a complete lattice (every intersection of subheaps is a subheap).

   \noindent (b) Congruences on  a heap form a complete lattice (every intersection of congruences is a congruence).

   \noindent (c) Normal heaps of a heap do not form a lattice in general, but only a partially ordered set, because the intersection of two normal subheaps can be empty.}\end{rem}

   \begin{exam} As an example, consider the heap $(\Z,[-,-,-])$ of integer numbers with $[a,b,c]=a-b+c$. The complete lattice of its subheaps is $\{\,a+b\Z\mid a,b\in\Z\,\}\cup\{\emptyset\}$. The set of its normal subheaps is $\{\,a+b\Z\mid a,b\in\Z\,\}$. Its congruences are the congruences $\equiv_n$ modulo $n$, and the complete lattice of congruence is $\{\,\equiv_n\mid n\in\N\,\}$, which is isomorphic to the lattice $(\N,|)$ with $0$ as its greatest element and $1$ as its least element.

   \smallskip

   Anticipating notions we will introduce later, in an abelian heap all non-empty subheaps are normal. There is an onto mapping $\{\,$normal subheaps$\,\}\to \{\,$congruences $\,\}$, $S\mapsto\sim_S$, that in our example $(\Z,[-,-,-])$ is the correspondence $a+b\Z\mapsto{}$congruence $\equiv_{|b|}$ modulo~$|b|$. This is an onto mapping, but is not a bijection. Of course, $a+b\Z=c+d\Z$ if and only if $|b|=|d|$ and $a\equiv_{|b|} c$. In Proposition~\ref{eq}, we will see that in order to get a one-to-one correspondence, that is, a bijection, it suffices to fix an element $e\in\Z$, and associate with any normal subheap $e+b\Z$ containing $e$ the congruence $\equiv_{|b|}$ modulo $|b|$.
\end{exam}

 It is well know that, for a group $G$, there is a lattice isomorphism between the lattice of all congruences on $G$ and the lattice of all normal subgroups of $G$. The  situation for a generic heap $(X,p)$ is the following: 

 \begin{prop}\label{eq}
    {\rm \cite[Proposition 2.10]{Bre}} Let $X$ be a heap and $e$ be a fixed element of $X$. Then there is a lattice isomorphism between the lattice of all congruences on the heap $X$ and the lattice of all normal subheaps of $X$ that contain $e$. It associates with any congruence $\sim$ the equivalence class  $[e]_\sim$ of $e$. Conversely, it associates with any normal subheap $S$ of $X$ with $e\in S$ the congruence $\sim_S$ on $X$ defined, for every $x,y\in X$, by $x\sim_Sy$ if there exists $s \in S$ such that
$[x, y,s] \in S$.
 \end{prop}

 Recall that a {\em congruence} on a heap $(X,[-,-,-])$ is an equivalence relation $\sim$ on the set $X$ such that $[x,y,z]\sim[x',y',z']$, for every $x,x',y,y',z,z'\in X$ such that $x\sim x'$, $y\sim y'$ and $z\sim z'$. By Proposition \ref{eq}, for any heap $X$, the set of all normal subheaps of $X$ is the set of all congruence classes $[e]_\sim$ where $e$ ranges in the set $X$ and $\sim$ ranges in the set of all congruences on the heap $X$.
 
 \bigskip

 For any two normal subheaps $S,T$ of a heap $X$, we have that $\sim_S{}\subseteq{}\sim_T$ if and only if, for every $x,y\in X$ and every $s \in S$ such that $[x, y,s] \in S$, there exists $t\in T$ such that $[x, y,t] \in T$ \cite[Definition 2.9 and Proposition 2.10]{Bre}.

 If $S$ is a normal subheap of a heap $X$, for every $s\in S$ the congruence class of $s$ modulo $\sim_S$ is $S$.

 \bigskip

 By Proposition \ref{eq}, the lattice of all congruences on a heap $X$ is isomorphic to the lattice of all normal subgroups of any of the groups $(X,b_x)$. In particular, the lattice of all congruences on a heap is a complete modular lattice.

\medskip

Another ``restatement'' of Proposition \ref{eq} is given in Theorem~\ref{bijp}. In a group $G$, the partially ordered set $\Cal N(G)$ of all normal subgroups of $G$ is order isomorphic to the partially ordered set $\Cal C(G)$ of all congruences of the group $G$: there is a bijection $\Cal N(G)\to\Cal C(G)$, $N\mapsto{}\sim_N$. In Theorem~\ref{bijp} we will show that the corresponding mapping $\Cal N(X)\to\Cal C(X)$, $N\mapsto{}\sim_N$ for a heap $X$ is only a surjective mapping, which induces a bijection $\Cal N(X)/\!\simeq{}\to\Cal C(G)$, where $\simeq$ is a suitable equivalence relation on $\Cal N(X)$.

Recall that to any pre-order $\preceq$ on a set $A$ there corresponds a pair $(\simeq, \le)$, where $\simeq$ is the equivalence relation on $A$ defined, for every $a,b\in A$, by $a\simeq b$ if $a \preceq b$ and $b\preceq a$, and $\le$ is the partial order on the quotient set $A/\!\simeq{}:=\{\,[a]_{\preceq}\mid a\in A\,\}$  defined, for every $a,b\in A$, by $[a]_{\preceq}\le [b]_{\preceq}$ if $a \preceq b$ \cite[Proposition~2.2]{Carmelo}.

 \begin{thm}\label{bijp} Let $X$ be a heap. On the set $\Cal N(X)$ of all normal subheaps of $X$ define a pre-order $\preceq$ setting, for all $M,N\in \Cal N(X)$, $M\preceq N$ if for every $x,y\in X$ and $s\in M$ such that $[x, y,s] \in M$ there exists $t\in N$ such that $[x, y,t] \in N$. Let $\simeq$ be the equivalence relation on 
     $\Cal N(X)$ associated to the pre-order $\preceq$. Then the partially ordered set $\Cal N(X)/\!\simeq$ is order isomorphic to the partially ordered set $\Cal C(X)$ of all congruences of the heap $X$. 
 \end{thm}

\begin{Proof}
    The equivalence relation ${}\simeq{}$ on 
     $\Cal N(X)$ is defined, for all $M,N\in \Cal N(X)$, by $M\simeq N$ if, for every $x,y\in X$, $$\exists s\in M\ \mbox{\rm s.~t.~}[x, y,s] \in M\Leftrightarrow\exists t\in N\ \mbox{\rm s.~t.~}[x, y,t] \in N.$$ Hence, in the notation of Proposition~\ref{eq}, $M\simeq N$ if and only if ${}\sim_M{}={}\sim_N{}$. 
     
     Let us prove that ${}\sim_M{}={}\sim_N{}$ if and only if, for every $e\in M$, there exists $f\in N$ such that $[f]_{\sim_N}$ coincides with the coset $f\cdot_eM$ of $f$ modulo $M$ in the group $(X,b_e)$. 
     
     Assume ${}\sim_M{}={}\sim_N{}$. Since $M$ and $N$ are normal subheaps, they are non-empty. Fix two elements $e\in M$ and $f\in N$. Let us prove that $[f]_{\sim_N}=f\cdot_eM$. 
     
     If $x\in [f]_{\sim_N}$, then $x\in X$ and $x\sim_N f$, so $x\sim_M f$. Thus there exists $m\in M$ such that $[x,f,m]\in M$. From \cite[Lemma~2.1(3)(b)]{Bre}, we get, in the group $(X,b_e)$, that $x\cdot_ef^{-1}\cdot_e m\in M$, and $M$ is a normal subgroup of $(X,b_e)$. Therefore $x\cdot_ef^{-1}\in M$, so $x\in M\cdot_ef=f\cdot_e M$.
     
     Conversely, suppose that $x\in M\cdot_ef=f\cdot_e M$. Then $x\cdot_ef^{-1}\in M$, so $[x,f,e]\in M$. Thus $x\sim_M f$, so $x\sim_N f$.
     
     For the inverse implication, suppose that $[f]_{\sim_N}=f\cdot_eM$, where $f\in N$. We must prove that ${}\sim_M{}={}\sim_N{}$. That is, that if $x,y\in X$, 
     $\exists s\in M$ s.~t.~$[x, y,s] \in M\Leftrightarrow\exists t\in N$ s.~t.~$[x, y,t] \in N$. 
    
     Fix two elements $x,y\in X$. Suppose that there exists $s\in M$ such that $[x, y,s] \in M$. Then, in the group $(X,b_e)$, we have that $x\cdot_ey^{-1}\cdot_es\in M$, so $x\cdot_ey^{-1}\in M$. Hence $x\cdot_ey^{-1}\cdot_ef\in M\cdot_ef=f\cdot_eM=[f]_{\sim_N}$. Thus $[x,y,f]\sim_Nf$, hence there exists $n\in N$ for which $[[x,y,f],f,n]\in N$, i.e., $[x,y,n]\in N$. 

     Conversely, suppose that $[x,y,t]\in N$ for some $t\in N$. Then $[[x,y,t],f,t]\in N$, from which $[x,y,t]\sim_Nf$. Therefore $[x,y,t]\in [f]_{\sim_N}=f\cdot_eM$. Equivalently, in the group $(X,b_e)$, we have that $x\cdot_ey^{-1}\cdot_ef\in f\cdot_eM=M\cdot_ef$, so $x\cdot_ey^{-1}\in M$. Thus $[x,e,[e,y,e]]\in M$, and therefore $[x,y,e]\in M$, as desired.
\end{Proof}

We leave to the reader the proof of the following easy result:

\begin{cor}\label{2.9} Let $X$ be a heap, $x$ a fixed element of $X$, $\sim$ a congruence on the heap $X$ and $N:=[x]_\sim$ the normal subgroup of the group $(X,b_x)$ corresponding to the congruence $\sim$ according to Proposition~\ref{eq}. The following conditions are equivalent for any pair $y,z$ of elements of $X$:

{\rm (a)} $y\sim z$;

{\rm (b)} The cosets $y\cdot_xN$ and $z\cdot_xN$ of the group $(X,b_x)$ coincide.
\end{cor}
%



\section{Commutators in a heap}  

 Now let us consider the problem of determining a natural notion of commutator for a heap.  Let $R$ and $S$ be two congruences on a heap $X$, and let $R\times_XS$ be the set of all triples $(x,y,z)\in X^3$ such that $xRy$ and $ySz$. Notice that $R\times_XS$ is a  subheap of $X^3$. A canonical {\em connector} between $R$ and $S$ (see \cite[Example 1.2]{BG}, \cite{Ped1} and \cite{Ped2}) is the mapping $$p\colon R\times_XS\to X$$ defined by $p(x,y,z)=[x,y,z]$ for every $(x,y,z)\in R\times_XS$, provided that $xS[x,y,z]$ and $[x,y,z]Rz$ for every $(x,y,z)\in R\times_XS$. The {\em commutator} of $R$ and $S$ is the smallest congruence $[R,S]$ on the heap $X$ such that $R\times_XS\to X/[R,S]$, $(x,y,z)\mapsto [\,[x,y,z]\,]_{[R,S]}$, is a heap morphism. That is, for every $x_1,y_1, z_1,x_2,y_2, z_2,x_3,y_3, z_3\in X$ such that $x_iRy_i$ and $y_iSz_i$ for all $i=1,2,3$, one has that $$[[x_1,y_1, z_1],[x_2,y_2, z_2],[x_3,y_3, z_3]]\,[R,S]\,[[x_1,x_2, x_3],[y_1,y_2, y_3],[z_1,z_2, z_3]].$$ 

Let us compute the commutator of two congruences $R,S$ on a heap $(X,p)$. Fix an element $e$ in $X$. Let $N:=[e]_R, M:=[e]_S$ be the normal subgroups of the group $(X,b_e)$ corresponding to $R$ and $S$ (Proposition~\ref{eq}). Then $$\{\,(x, y, z)\mid (x, y)\in R\ \mbox{\rm and }(y, z)\in S\,\}=\{\,(ny, y, ym)\mid y\in X, n\in N, m\in M\,\},$$ where, to simplify notation, we have denoted by justapposition the multiplication in the group $(X,b_e)$. Let $C:=[e]_T$ be the normal subgroup corresponding to an arbitrary congruence $T$ of $(X,p)$. We must determine when the mapping $p\colon \{\,(ny, y, ym)\mid y\in X, n\in N, m\in M\,\}\to X/C$, $p(ny, y, ym)=nymC$, is a heap morphism. This occurs if and only if $$p(p(n_1y_1, y_1, y_1m_1),p(n_2y_2, y_2, y_2m_2),p(n_3y_3, y_3, y_3m_3))=n_1y_1m_1(n_2y_2m_2)^{-1}n_3y_3m_3C$$ coincides with $$p(p(n_1y_1,n_2y_2,n_3y_3),p( y_1, y_2,y_3),p(y_1m_1, y_2m_2, y_3m_3))=$$ $$=n_1y_1(n_2y_2)^{-1}n_3y_3(y_1y_2^{-1}y_3)^{-1}y_1m_1(y_2m_2)^{-1}y_3m_3C.$$ Equivalently, if and only if $$m_1(n_2y_2m_2)^{-1}n_3C=(n_2y_2)^{-1}n_3y_3(y_1y_2^{-1}y_3)^{-1}y_1m_1(y_2m_2)^{-1}C$$ for every $m_1, m_2\in M,$ $ y_2\in X$, $n_2,n_3\in M$. This can be rewritten as \begin{equation}
n_3^{-1}n_2y_2m_2m_1^{-1}y_2^{-1}n_2^{-1}n_3y_3y_3^{-1}y_2y_1^{-1}y_1m_1m_2^{-1}y_2^{-1}\in C.  \label{ct7i}  
\end{equation} Replacing $n_3^{-1}n_2$ with $n\in N$, $y_2$ with $y\in X$ and $m_2m_1^{-1}$ with $m\in M$, the previous condition (\ref{ct7i}) becomes $$nymy^{-1}n^{-1}ym^{-1}y^{-1}\in C$$ for every $m\in M, y\in X$, $n\in M$, that is, $$y^{-1}nymy^{-1}n^{-1}ym^{-1}\in C.$$ Let us prove that this is equivalent to the condition $[N,M]\subseteq C$. If $$y^{-1}nymy^{-1}n^{-1}ym^{-1}\in C$$  for every $m\in M, y\in X$, $n\in M$, we get, for $y=1_{(X,b_e)}=e$, that $nmn^{-1}m^{-1}\in C$, so that $[N,M]\subseteq C$. Conversely, if $[N,M]\subseteq C$, $m\in M, y\in X$ and $n\in M$, then $$y^{-1}nymy^{-1}n^{-1}ym^{-1}=[y^{-1}ny,m]\in[N,M]\in C.$$ We have thus proved that:

\begin{thm}\label{4.1} Let $R$ and $S$ be two congruences on a heap $(X,p)$. Fix an element $e$ in $X$. Let $N:=[e]_R$ and $M:=[e]_S$ be the normal subgroups of the group $(X,b_e)$ corresponding to the congruences $R$ and $S$ respectively. Then the commutator $[R,S]$ of $R$ and $S$ is the congruence on $(X,p)$ corresponding to the normal subgroup $[N,M]$ of the group $(X,b_e)$.\end{thm}

In particular, a heap $(X,p)$ is {\em abelian} if and only if $[X,X]=\{ e\}$ in the group $(X,b_e)$, that is, if and only if the group $(X,b_e)$ is abelian. Since all the groups $(X,b_y)$ are isomorphic, this is equivalent to all the groups $(X,b_y)$ being abelian, that is, $[x,y,z]=[z,y,x]$ for every $x,y,z\in X$. 

\medskip

Notice that the variety of heaps is congruence-modular (paragraph before the statement of Corollary~\ref{2.9}), so that our notion of commutator is particularly good (see \cite{15}). We will come back to commutators of congruences in heaps in dealing with  commutators of congruences in trusses (Section~\ref{C}).

\begin{rem}\label{2.2}{\rm \cite[Remark 2.2]{Brzez} In \cite[Lemma~2.3(3)]{Bre}, it is shown that, for any heap $X$, $[[a,b,c],d,e]=[a,[d,c,b],e]=[a,b,[c,d,e]]$ for every $a,b,c,d,e\in X$. Hence in an abelian heap $X$, $ [[a,b,c],d,e]=[a,[b,c,d],e]=[a,b,[c,d,e]]$. This can be generalized as follows.

In dealing with heaps, that is, with a unique ternary operation $[-,-,-]$, all terms consists of an odd number of occurences of variables $[a_1,a_2,\dots,a_{2n+1}]$ with a suitable placement of brackets. In an abelian heap $(H,[-,-,-])$, the placement of brackets in multiple applications of the heap operation does not play any role, hence it is possible to write $[a_1,a_2,\dots,a_{2n+1}]$ for any such multiple applications. In such an abelian heap, the parity of the position of an element does matter, but any element within an even or odd position in the operation may exchange position with any other element in a respectively even or odd position. Moreover, if after such a parity preserving rearrangement two adjacent elements are equal to each other, it is possible to cancel these two equal consecutive elements.} \end{rem}

\section{Idempotent endomorphisms and semidirect products of heaps}

In any algebraic structure, idempotent endomorphisms are related to semidirect products. Let us see what occurs for semidirect products as far as heaps are concerned. In this section we will always suppose that $X$ is a non-empty heap (for $X=\emptyset$, the unique idempotent endomorphism $f\colon X\to X$ is the identity mapping of $X$).

\begin{prop}\label{boh} Let $X\ne\emptyset$ be a heap, $Y$ be a subheap of $X$, and $\omega$ a congruence on $X$. The following conditions are equivalent:

{\rm (a)} $Y$ is a set of representatives of the equivalence classes of $X$ modulo $\omega$, that is, $Y\cap[x]_\omega$ is a singleton for every $x\in X$.

{\rm (b)} There exists an idempotent heap endomorphism of $X$ whose image is $Y$ and whose kernel is $\omega$.

{\rm (c)} For every $e\in Y$, there exists an idempotent group endomorphism of the group $(X,b_e)$ whose image is the subgroup $Y$ of $(X,b_e)$ and whose kernel is the normal subgroup $[e]_\omega$ of $(X,b_e)$. 

{\rm (d)} There exist an element $e\in Y$ and an idempotent group endomorphism of the group $(X,b_e)$ whose image is the subgroup $Y$ of $(X,b_e)$ and whose kernel is the normal subgroup $[e]_\omega$ of $(X,b_e)$. 

{\rm (e)} For every $e\in Y$, the group $(X,b_e)$ is the semidirect product of its subgroup $Y$ and its normal subgroup $[e]_\omega$. 

{\rm (f)} There exist an element $e\in Y$ such that the group $(X,b_e)$ is the semidirect product of its subgroup $Y$ and its normal subgroup $[e]_\omega$. 

{\rm (g)}
For every $a\in X$ and every $c\in Y$ there exist a unique element $b\in X$ and a unique element $d\in Y$ such that $a=p(b,c,d)$ and $b\,\omega\, c$.

  {\rm (h)}
For every $a\in X$ and every $c\in Y$ there exist a unique element $b\in X$ and a unique element $d\in Y$ such that $a=p(d,c,b)$ and $b\,\omega\, c$.

{\rm (i)} The mapping $g\colon Y\to X/\omega$, defined by $g(y)=[y]_\omega$ for every $y\in Y$, is a heap isomorphism.

{\rm (l)} There exists a heap endomorphism of $X$ whose kernel is $\omega$ and whose restriction to $Y$ is the inclusion of $Y$ in $X$.
\end{prop} 

\begin{proof}
    (a)${}\Rightarrow{}$(b) Suppose that (a) holds. Let $f\colon X\to X$ be the mapping that associates with each $x\in X$ the unique element of the singleton $Y\cap[x]_\omega$. In order to show that $f$ is a heap endomorphism, fix $x,y,z\in X$. Let $x',y',z'$ be the unique elements of  $Y\cap[x]_\omega$, $Y\cap[y]_\omega$, $Y\cap[z]_\omega$ respectively. Then $f(x)=x'$, $f(y)=y'$, $f(z)=z'$, and $[x',y',z']\in Y$ because $Y$ is a subheap. Moreover $x\,\omega\, x'$, $y\,\omega\, y'$ and $z\,\omega\, z'$, hence $[x,y,z]\,\omega\,[x',y',z']$. Thus $[x',y',z']\in Y\cap [[x,y,z]]_\omega$, so $f([x,y,z])=[x',y',z']=[f(x),f(y),f(z)]$.

It remains to show that $f$ is idempotent. If $x\in X$, let $x'$ be the unique element of $Y\cap[x]_\omega$. Then $f(x)=x'$ and $x'\in Y\cap [x']_\omega$, so $f(x')=x'$. This proves that $f(f(x))=f(x)$.

(b)${}\Rightarrow{}$(c) Let $f\colon X\to X$ be an idempotent heap endomorphism with $f(X)=Y$ and kernel $\omega$. Let $e$ be an element of $Y$. Then $f(e)=e$, so that $f$ is also a group endomorphisms of the group $(X,b_e)$. Its kernel is $f^{-1}(e)= [e]_\omega$.

(c)${}\Rightarrow{}$(d) follows from the fact that $Y\ne\emptyset$, because $X\ne\emptyset$ and $Y=f(X)$.

(d)${}\Rightarrow{}$(b) If (d) holds and $f$ is an idempotent group endomorphism of $(X,b_e)$ whose image is $Y$ and whose kernel is $[e]_\omega$, then $f$ is the required heap endomorphism of $X$. 

(b)${}\Rightarrow{}$(a) Let $f$ be an idempotent heap endomorphism of $X$ whose image is $Y$ and whose kernel is $\omega$. Let us prove that $Y\cap[x]_\omega=\{f(x)\}$ for every $x\in X$. If $y\in Y\cap[x]_\omega$, then $y=f(z)$ for some $z\in X$, and $y\,\omega\,x$. Thus $y=f(z)=f(f(z))=f(y)=f(x)$. This proves that $Y\cap[x]_\omega\subseteq\{f(x)\}$. The reverse inclusion is trivial.

(c)${}\Leftrightarrow{}$(e) and (d)${}\Leftrightarrow{}$(f) are well known.

(e)${}\Leftrightarrow{}$(g) and (e)${}\Leftrightarrow{}$(h) are just a restatement of the fact that a group $G$ is the semidirect product of a normal subgroup $N$ and a subgroup $H$ if and only if every element of $G$ can be written in a unique way in the form $nh$, if and only if every element of $G$ can be written in a unique way in the form $hn$. For instance, suppose that (e) holds. To prove (g), fix two elements $a\in X$ and $c\in Y$. The group $(X,b_c)$ is the semidirect product of $Y$ and $[c]_\omega$. Thus the element $a$ of $X$ can be written in a unique way as a product of an element of $[c]_\omega$ and an element of $Y$, that is, there exist a unique $b\in [c]_\omega$ and a unique $d\in Y$ such that $a=b_c(b,d)$. Equivalently, there exist a unique $b\in X$ and a unique $d\in Y$ such that $a=p(b,c,d)$ and $b\,\omega\, c$. 

(a)${}\Leftrightarrow{}$(i) The mapping $g\colon Y\to X/\omega$, defined by $g(y)=[y]_\omega$ for every $y\in Y$, is always a heap homomorphism. It is a bijective mapping if and only if for every element $x$ of $X$ there is a unique element $y\in Y$ congruent to $x$ modulo $\omega$, that is, if and only if (a) holds.

(i)${}\Rightarrow{}$(l) If $\varepsilon\colon Y\to X$ is the inclusion, $\pi\colon X\to X/\omega$ is the canonical projection, and $g=\pi\varepsilon$ is a heap isomorphism, then the composite mapping  $((\pi\varepsilon)^{-1}\pi)\varepsilon$  is the identity of $Y$, hence $\varepsilon((\pi\varepsilon)^{-1}\pi)\varepsilon$ is the inclusion of $Y$ in $X$. Therefore $\varepsilon(\pi\varepsilon)^{-1}\pi$ is a heap endomorphism of $X$ whose kernel is $\omega$ and whose restriction to $Y$ is the inclusion of $Y$ in $X$.

We leave (l)${}\Rightarrow{}$(i) as an exercise to the reader. \end{proof}

If $X$ is a heap, and its subheap $Y$ and the congruence $\omega$ on $X$ satisfy the equivalent conditions of Proposition \ref{boh}, we will say that $X$ is the {\em (inner) semidirect product} of $\omega$ and $Y$, and write $X=\omega\rtimes Y$. It is now easy to prove that:

\begin{thm}
    Let $X$ be a heap. Then there is a one-to-one correspondence between:

    {\rm (a)} The set $A$ of all idempotent heap endomorphisms of $X$, that is, the set of all $f\in\End_{\Hp}(X)$ such that $f^2=f$.

     {\rm (b)} The set $B$ of all pairs $(\omega,Y)$, where $\omega$ is a congruence on $X$, $Y$ is a subheap of $X$, and $X=\omega\rtimes Y$.

Given an idempotent heap endomorphism $f\in A$, the corresponding pair in $B$ is $$(\ker(f), f(X)).$$ Conversely, if $(\omega,Y)\in B$, the idempotent endomorphism $f$ of $X$ is defined by $Y\cap[x]_\omega=\{f(x)\}$ for every $x\in X$.
\end{thm}

Now let $X$ be a heap, $\omega$ be a congruence on $X$, $Y$ be a subheap of $X$, and suppose $X=\omega\rtimes Y$. Then, for every fixed element $e\in Y$, the group $(X,b_e)$ is the semidirect product of its subgroup $Y$ and its normal subgroup $[e]_\omega$ (Proposition~\ref{boh}), so that 
the group $(Y,b_e)$ acts over the group $K:=[e]_\omega$, via conjugation: there is a group homomorphism $\alpha\colon (Y,\beta_e)\to(\Aut_{\Gp}(K),\circ)$, $\alpha\colon y\in Y\mapsto \alpha_y$, where $$\alpha_y(k)=yky^{-1}=[[y,e,k],e,[e,y,e]]]=[y,e,[k,y,e]]$$
for every $y\in Y$, $k\in K$.
 
\bigskip

We leave to the reader the easy proof of the following result:

\begin{prop}\label{vhl} Let $f$ be an idempotent heap endomorphism of a heap $(X,[-,-,-])$, let $\omega$ be the kernel of $f$, so that $\omega$ is a congruence of the heap $X$, and let $e$ be a fixed element of $Y:=f(X)$. Then $[e]_\omega=f^{-1}(e)$ and there is an action, i.e., a heap morphism, $$\alpha\colon (Y,[-,-,-])\to \Aut_{\Hp}([e]_\omega),$$ defined by $\alpha_y(k)=[y,e,[k,y,e]]$ for every $y\in Y$ and $k\in [e]_\omega$ . Moreover, $\alpha_e$ the identity automorphism of $[e]_\omega$.
\end{prop}

The operation $[-,-,-]$ on $\Aut_{\Hp}([e]_\omega)$ is defined by $[f,g,h]=f\circ g^{-1}\circ h$ for every $f,g,h\in \Aut_{\Hp}([e]_\omega)$. 
All computations necessary to check the statement of Proposition~\ref{vhl} follow easily making use of the techniques presented in Remark~\ref{2.2}.

\bigskip

We are ready to define the outer semidirect product of heaps:

\begin{prop} Let $K$ and $Y$ be heaps. Let $\alpha\colon Y\to \Aut_{\Hp}(K),$ $\alpha\colon b\in Y\mapsto \alpha_b$, be a heap morphism, and let $y$ be an element of $Y$ that is mapped by $\alpha$ to the identity automorphism $\alpha_y$ of $K$. On the cartesian product $K\times Y$ define a ternary operation $[-,-,-]$ setting $$[(k_1,y_1),(k_2,y_2),(k_3,y_3)]:=([k_1,\alpha_{[y_1,y_2,y]}(k_2), \alpha_{[y_1,y_2,y]}(k_3)], [y_1,y_2,y_3])$$ for every $(k_1,y_1),(k_2,y_2),(k_3,y_3)\in K\times Y$. Then:

{\rm (a)} $K\times Y$ is a heap and $K\times\{y\}$ is a normal subheap of $K\times Y$ isomorphic to $K$. 

{\rm (b)} For every element $k\in K$, $\{k\}\times Y$ is a subheap of $K\times Y$ isomorphic to $Y$ and the mapping $f\colon K\times Y\to K\times Y$ defined by $f(a,b)=(k,b)$ for every $(a,b)\in K\times Y$ is an idempotent heap endomorphism of $K\times Y$.
\end{prop}

\begin{Proof}
{\rm (a)} Let us show that $K\times Y$ is a heap. For the Mal'tsev identities we have:
$$\begin{array}{l}
[(k_1,y_1),(k_1,y_1),(k_2,y_2)]=
([k_1,\alpha_{[y_1,y_1,y]}(k_1), \alpha_{[y_1,y_1,y]}(k_2)], [y_1,y_1,y_2 ] )=\\
\quad=([(k_1,\alpha_y(k_1),\alpha_y (k_2)],y_2) =([(k_1,k_1,k_2],y_2)=(k_2,y_2),\end{array}$$
and $$\begin{array}{l}
[(k_1,y_1),(k_2,y_2),(k_2,y_2)]=
([k_1,\alpha_{[y_1,y_2,y]}(k_2), \alpha_{[y_1,y_2,y]}(k_2)], [y_1,y_2,y_2])=\\
\quad=[k_1,y_1].\end{array}$$

Let us check associativity: 
\noindent$$\begin{array}{l}
[[(k_1,y_1),(k_2,y_2),(k_3,y_3)],(k_4,y_4),(k_5,y_5)]=\\
\left[ ( [ k_1 ,\alpha_{[y_1,y_2,y]}(k_2), \alpha_{[y_1,y_2,y]}(k_3)], [y_1,y_2,y_3]),(k_4,y_4),(k_5,y_5)\right]=\\
\left( 
 [[ k_1 ,\alpha_{[y_1,y_2,y]}(k_2), \alpha_{[y_1,y_2,y]}(k_3)],
 \alpha_{\left[ [y_1,y_2,y_3],y_4,y\right]}(k_4),
 \alpha_{\left[ [y_1,y_2,y_3],y_4,y \right]}(k_5)],\left[ [y_1,y_2,y_3 ],y_4,y_5 
\right] \right)=\\
\left( \left[ 
 [ k_1 ,\alpha_{[y_1,y_2,y]}(k_2), \alpha_{[y_1,y_2,y]}(k_3)],
 \alpha_{\left[ y_1,y_2,[y_3,y_4,y]\right]}(k_4),
 \alpha_{\left[ y_1,y_2,[y_3,y_4,y]\right]}(k_5)],\left[ y_1,y_2,[ y_3,y_4,y_5 
\right]\right]   \right ) =\\
\left( \left[ 
 k_1 ,\alpha_{[y_1,y_2,y]}(k_2), [\alpha_{[y_1,y_2,y]}(k_3),
 \alpha_{\left[ y_1,y_2,[y_3,y_4,y]\right]}(k_4),
 \alpha_{\left[ y_1,y_2,[y_3,y_4,y]\right]}(k_5)]],\left[ y_1,y_2, [y_3,y_4,y_5 
\right] \right] \right).
\end{array}$$

\smallskip

On the other hand,  

\noindent$\begin{array}{llllllllll}
\left[ (k_1,y_1),(k_2,y_2),[(k_3,y_3),(k_4,y_4),(k_5,y_5) ]\right]=&\\
\left[
(k_1,y_1),(k_2,y_2),( [k_3,\alpha_{[y_3,y_4,y]}(k_4), \alpha_{[y_3,y_4,y]}(k_5)], [y_3,y_4,y_5 ] ) \right]=&\\
 \left(
[ k_1,\alpha_{[y_1,y_2,y]}(k_2),\alpha_{[y_1,y_2,y]}([k_3,\alpha_{[y_3,y_4,y]}(k_4), \alpha_{[y_3,y_4,y]}(k_5)])],[y_1,y_2, [ y_3,y_4,y_5 ] ] \right)=&\\
 \left(
[ k_1,\alpha_{[y_1,y_2,y]}(k_2),[\alpha_{[y_1,y_2,y]}(k_3),\alpha_{[y_1,y_2,y]}\alpha_{[y_3,y_4,y]}(k_4),\alpha_{[y_1,y_2,y]} \alpha_{[y_3,y_4,y]}(k_5)]],\right. \\ \qquad \left. [y_1,y_2[ y_3,y_4,y_5 ] ] \right)=&\\
\left( \left[ k_1,\alpha_{[y_1,y_2,y]}(k_2),\right. \right. &\\
\qquad\left[ (\alpha(y_1)\circ \alpha(y_2)^{-1}\circ \alpha(y))(k_3),(\alpha(y_1)\circ \alpha(y_2)^{-1}\circ \alpha(y))\circ (\alpha(y_3)\circ \alpha(y_4)^{-1}\circ \alpha(y))(k_4)),\right. \\
\qquad(\alpha(y_1)\circ \alpha(y_2)^{-1}\circ \alpha(y))\circ (\alpha(y_3)\circ \alpha(y_4)^{-1}\circ \alpha(y))(k_5))] ],
\left. \left[ y_1,y_2[ y_3,y_4,y_5 ] \right] \right)= & \\
\left( \left[ k_1,\alpha_{[y_1,y_2,y]}(k_2),\right. \right.&\\
\qquad
 [(\alpha(y_1)\circ \alpha(y_2)^{-1}\circ \alpha(y))(k_3),
(\alpha(y_1)\circ \alpha(y_2)^{-1}\circ  \alpha(y_3)\circ \alpha(y_4)^{-1}\circ \alpha(y))(k_4), \\
\qquad\left. \alpha(y_1)\circ \alpha(y_2)^{-1}\circ \alpha(y_3)\circ \alpha(y_4)^{-1}\circ \alpha(y))(k_5)], [y_1,y_2[ y_3,y_4,y_5 ] ]\right)=&\\
( [ k_1,\alpha_{[y_1,y_2,y]}(k_2),
[\alpha_{[y_1,y_2,y]}(k_3),
\alpha_{[y_1,y_2,[y_3,y_4,y]]}(k_4)),
\alpha_{[y_1,y_2,[y_3,y_4,y]]}(k_5) ] ],[ y_1,y_2[ y_3,y_4,y_5 ] ] ).
\end{array}$

\bigskip

Now let us prove that $K \times \{y \}$ is a normal subheap of $K \times Y$: 
for every $(k_1,y),(k_2,y)$, $(k_3,y) \in K \times \{y\},$ we have that $$[(k_1,y),(k_2,y),(k_3,y)]= ([k_1, \alpha_{[y,y,y]}(k_2), \alpha_{[y,y,y]}(k_3)], [y,y,y])=([k_1,k_2, k_3],y).$$
By Lemma~\ref{3.3}(c) it is enough to show that 
for every $(k, y_1) \in K \times Y$ and $ (k_1,y), (k_2,y) \in K \times \{y\} $
we have  that $ [[(k,y_1), (k_1,y), (k_2,y)], (k,y_1),(k_1,y)]\in K \times \{y\}.$

Now $$\begin{array}{l}
 \left[ [(k,y_1), (k_1,y), (k_2,y)], (k,y_1),(k_1,y)\right]=\\
 \left[ ( [k, \alpha_{[y_1,y,y]}(k_1), \alpha_{[y_1,y,y]}(k_2)], [y_1,y,y] ), (k,y_1),(k_1,y)\right]=\\
 \left[ ( [k, \alpha_{y_1}(k_1), \alpha_{y_1}(k_2)], y_1 ),(k,y_1),(k_1,y)\right]=\\
 ( \left[ [ k, \alpha_{y_1}(k_1), \alpha_{y_1}(k_2)],\alpha_{[y_1,y_1,y]}(k),\alpha_{[y_1,y_1,y]}(k_1)],[y_1,y_1,y]\right] )=\\
 ( [ [ k, \alpha_{y_1}(k_1), \alpha_{y_1}(k_2)],\alpha_{y}(k),\alpha_{y}(k_1)],y)\in K\times \{y\}. 
    \end{array}$$
We leave the rest of the proof to the reader.
\end{Proof}

It is clear how to define the (direct) product of two heaps. It corresponds to direct products of groups.

\begin{prop}\label{5.5} Let $X$ be a heap, $\omega$ be a congruence on $X$, $Y$ be a subheap of $X$, and suppose $X=\omega\rtimes Y$. Fix an element $e\in Y$ and set $K:=[e]_\omega$. Let $f$ be the corresponding idempotent heap morphisms of $X$.
Then the following conditions are equivalent:

{\rm (a)} The subheap $Y$ of $X$ is normal.
    
{\rm (b)} $X= K\times Y$ (direct product of heaps).

{\rm (c)} 
$[y,e,k]=[k,e,y]$ for every $y\in Y$ and every $k\in K$. 

{\rm (d)} There is an idempotent heap endomorphism $g$ of $X$ whose image is $K$ and such that $g^{-1}(e)=Y$.

{\rm (e)} The heap morphism $\alpha\colon Y\to \Aut_{\Hp}(K)$ is constantly equal to the identity mapping of $K$.
\end{prop}

\begin{Proof} (a)${}\Leftrightarrow{}$(b)${}\Leftrightarrow{}$(c) For a semidirect-product decomposition of a group $G$ as a semidirect product of a subgroup $H$ of $G$ and a normal subgroup $N$ of $G$, one has that $G$ is the direct product of $H$ and $N$ if and only if $H$ is also normal in $G$, if and only if the elements of $H$ commute with the elements of $N$. Applying this to the group $(X,b_e)$, one gets that (b)${}\Leftrightarrow{}$(a)${}\Leftrightarrow{}$(c).

(a)${}\Rightarrow{}$(d) If $Y$ is a normal subheap of $X$, let $g$ be the idempotent heap endomorphism of $X$ whose image is $K$ and whose kernel is $\sim_Y$.

(d)${}\Rightarrow{}$(a) If $g$ is an idempotent heap endomorphism of $X$ for which $g^{-1}(e)=Y$, then $Y$ is a normal subheap of $X$.

(c)${}\Leftrightarrow{}$(e) is trivial
    \end{Proof}

In particular, if the equivalent conditions of the previous Proposition~\ref{5.5} hold, then the commutator $[\sim_Y,\omega]$ is the conguence $=$ on $X$. 

\section{Left near-trusses}

We will now follow the content of \cite[Section~2]{Brebetween}. A {\em left near-truss} $(X,[-,-,-],\cdot{})$ is a set $X$ endowed with a ternary operation $[-,-,-]$ and a binary operation $\cdot$, such that $(X,[-,-,-])$ is a heap, $(X,\cdot{})$ is a semigroup, and {\em  left distributivity} holds, that is, $$x\cdot[y,z,w]=[x\cdot y,x\cdot z,x\cdot w]$$  for every $x,y,z,w\in X$. Similarly for {\em right near-trusses}, where left distributivity is replaced by {\em right distributivity}: $[y,z,w]\cdot x=[y\cdot x,z\cdot x,w\cdot x]$ for every $x,y,z,w\in X$. Clearly, the category of  left near-trusses is isomorphic to the category of right near-trusses, it suffices to associate to any left near-truss $(X,[-,-,-],\cdot{})$ its opposite right near-truss $(X,[-,-,-],\cdot{}^{\op})$.

\begin{exams}\label{6.2*} (1) Let $(X,[-,-,-])$ be a heap and let $$M(X):=\{\,f\mid f \colon X\to X\,\}$$ be the set of all mappings from the set $X$ to itself. Define a ternary operation $[-,-,-]$ on $M(X)$ setting, for every $f,g,h\in M(X)$, $[f,g,h](x)=[f(x),g(x),h(x)]$ for all $x\in~X$. Then $(M(X),[-,-,-])$ is also a heap (it is the direct product of $|X|$ copies of the heap $(X,[-,-,-])$). Taking the composition of mappings as the binary operation $\cdot$, $M(X)$ becomes a {\em right} near-truss. 

(2) Let $(N,+,\cdot{})$ be a left near-ring. Define a ternary operation $[-,-,-]\colon N\times N\times N\to N$ on $N$ setting $[x,y,z]=x-y+z$ for every $x,y,z\in N$. Then $(N,[-,-,-],\cdot{})$ is a left near-truss. 

(3) Let $(B,*,\circ)$ be a left skew brace \cite{BFP}. Define a ternary operation $[-,-,-]\colon B\times B\times B\to B$ on $B$ setting $[x,y,z]=x*(y^{-*})*z$ for every $x,y,z\in B$. Then $(B,[-,-,-],\circ)$ is a left near-truss. \begin{Proof} In order to show that $(B,[-,-,-],\circ)$ is a left near-truss it suffices to check left distributivity. That is, that $a\circ[b,c,d]=a\circ(b*(c^{-*})*d)=(a\circ b)*(a^{-*})*(a\circ (c^{-*}))*(a^{-*})*(a\circ d)$ is equal to $[a\circ b,a\circ c,a\circ d]=(a\circ b)*(a\circ c)^{-*}*(a\circ d)$. This is equivalent to proving that $(a^{-*})*(a\circ (c^{-*}))*(a^{-*})=(a\circ c)^{-*}$. Now $(a^{-*})*(a\circ (c^{-*}))*(a^{-*})=\lambda_a(c^{-*})*(a^{-*})=(\lambda_a(c))^{-*}*(a^{-*})=(a*\lambda_a(c))^{-*}=(a\circ c)^{-*}$, as we wanted to show.
    \end{Proof}\end{exams}

The right near-truss $M(X)$ of Example  \ref{6.2*}(1) is particularly interesting because:

\begin{thm}
    Every right near-truss is isomorphic to a subnear-truss of $M(X)$ for some heap $X$.
\end{thm}

\begin{Proof} Let $(Y,[-,-,-] ,\cdot)$ be a right near-truss and $(X,[-,-,-], \circ)$ be a right near-truss  properly containing the right near truss $(Y,[-,-,-], \cdot )$. 
For example $X$ could be the direct product $Y\times\Z$ with the operations $[(y,z),(y'z'),(y'',z'')]=([y,y',y''],z-z'+z''),$ $(y,z)\circ (y',z')=(y\cdot y',z\cdot z').$

  Consider the mapping $\lambda\colon (Y,[-,-,-],\cdot)\to (M(X),[-,-,-],\circ)$ defined by $\lambda\colon y\mapsto \lambda_{(y,1)}$, where $\lambda_{(y,1)}\colon X\to X$ is defined, for every $(y_1, z)\in X$, by

$$\lambda_{(y,1)}\colon (y_1,z)\mapsto \left\{\begin{array}{ll} (y\cdot y_1 ,1) & \mbox{\rm if $(y_1,z)\in Y\times \{1 \}$,}  \\
(y,1) & \mbox{\rm if $(y_1,z)\in  X\setminus (Y\times \{ 1\} )$.}
\end{array}\right.$$

It is easily checked that the mapping $\lambda$ is injective and is a right near-truss morphism.
\end{Proof}

\begin{thm}\label{6.2} Let $(X,[-,-,-],\cdot{})$ be a left near-truss, and fix an element $y\in X$. Then $(X,b_y,\cdot{})$ is an algebra (in the sense of Universal Algebra) in which $(X,b_y)$ is a group $(X,*_y)$, $(X,\cdot{})$ is a semigroup, and $w(x*_yz)=(wx)*_y(wy)^{-*}*_y(wz)$ for every $x,y,z,w\in X$. Here $(wy)^{-*}$ denotes the inverse of the element $w\cdot y$ in the group $(X,b_y)=(X,*_y)$.\end{thm} 

{\sc Proof.}
    $$\begin{array}{l}w(x*_yz)=w[x,y,z]=[wx,wy,wz]=[[wx,y,y], wy,[y,y,wz]]= \\ \qquad =[wx,y,[y, wy,[y,y,wz]]]=[wx,y,[[y, wy,y],y,wz]]= \\ \qquad =(wx)*_y((wy)^{-*}*_y(wz))=(wx)*_y(wy)^{-*}*_y(wz).\qquad\qquad\rule{1ex}{1ex}\end{array}$$

\bigskip

For every $x,y$ in a left near-truss $(X,[-,-,-],\cdot{})$, left multiplication by $x$, $\lambda_x\colon X\to X$, $\lambda_x\colon z\in X\to x\cdot z$, is a group morphism $\lambda_x\colon (X,b_y)\to (X,b_{xy})$.

If we write the operation in the group  $(X,b_y)$ additively, the equality given by
left distributivity $x\cdot[z,y,u]=[x\cdot z,x\cdot y,x\cdot u]$ becomes 
$x(z+u)=xz-xy+xu$ for every $x,z,u\in X$. The two most interesting cases of this equality are those for which $xy=0_{(X,b_y)}=y$ for every $x\in X$, and $xy=x$ for every $x\in X$. It easily follows that:

\begin{Lemma}\label{6.3} Let $(X,[-,-,-],\cdot{})$ be a left near-truss and $y$ be a fixed element of $X$.

{\rm (a)} If $y$ is a right zero for the semigroup $(X,\cdot{})$ (that is, $xy=y$ for every $x\in X$), then $(X, b_y,\cdot )$ is a left near-ring.

{\rm (b)} If $(X,\cdot{})$ is a group and $y$ is its identity, then $(X, b_y,\cdot{})$ is a left skew brace.\end{Lemma}

\begin{exam}\label{466} Let $(R,+,\cdot{})$ be any ring with identity. Define $[x,y,z]=x-y+z$. Then $(R, [-,-,-],\cdot{})$ is a left near-truss, $(R,b_0,\cdot{})\cong (R,
+,\cdot{})$ is a ring, and $(R,b_1,\cdot{})\cong (R,+,\circ)$, where $\circ$ denotes the Jacobson multiplication on the ring $R$, defined by $x\circ y=x+y+xy$ for every $x,y\in R$.\end{exam}

The appearance of left skew braces into the picture in Example~\ref{6.2*}(3) and in Lemma~\ref{6.3} suggests the existence, for a left near-truss $(X,[-,-,-],\cdot{})$, of an action of the semigroup $(X,\cdot{})$ on each group $(X,b_y)$, that is, of a natural semigroup morphism $$\lambda\colon (X,\cdot{}) \to\End_{\Gp}(X, b_y)$$ into the endomorphism semigroup $\End_{\Gp}(X, b_y)$ of the group $(X, b_y)$. More precisely:

\begin{Lemma}\label{Bachi} Let $(X,[-,-,-],\cdot{})$ be a left near-truss and $y$ be a fixed element of $X$. Denote by $*$ the operation in the group $(X,b_y)$ and by $z^{-*}$ the inverse of any element $z\in X$ in the group $(X,b_y)$. Then $\lambda^y\colon (X,\cdot{})\to\End_{\Gp}(X, b_y)$, given by $\lambda^y\colon x\mapsto\lambda_x^y$, where $\lambda_x^y(z)=(xy)^{-*}* (xz)$ for every $x,z\in X$, is a semigroup morphism.\end{Lemma}

\begin{Proof} Left distributivity $x[z, y,u]=[xz, xy,xu]$ can be rewritten, in the group $(X,b_y)$, as $x(z*u)=(xz)*(xy)^{-*}* (xu)$. Multiplying by $(xy)^{-*}$ on the left, we get that $$(xy)^{-*}*x(z*u)=(xy)^{-*}*(xz)*(xy)^{-*}* (xu).$$ In the notation of the statement of the Lemma, this identity can be rewritten as $\lambda_x^y(z*u)=\lambda_x^y(z)*\lambda_x^y(u)$, so that each $\lambda_x^y$ is a group endomorphism of $(X,b_y)$. Hence $\lambda^y$ is a mapping of $X$ into $\End_{\Gp}(X, b_y)$. To conclude the proof, we must show that $\lambda^y$ is a semigroup morphism, that is, that $\lambda^y_{xt}=\lambda^y_{x}\circ \lambda^y_{t}$. Now $\lambda^y_{xt}(z)=(xty)^{-*}* (xtz)$ and $$\begin{array}{l}
(\lambda^y_{x}\circ \lambda^y_{t})(z)=\lambda^y_{x}((ty)^{-*}* (tz))= \\ \qquad =(xy)^{-*}* (x((ty)^{-*}* (tz)))=(xy)^{-*}* (x(ty)^{-*})*(xy)^{-*}* (xtz).\end{array}\label{13}$$ We have that $xy=x((ty)*(ty)^{-*})=xty*(xy)^{-*}*x(ty)^{-*}$, from which \begin{equation}(xty)^{-*}*xy=(xy)^{-*}*x(ty)^{-*}.\label{12}\end{equation} Replacing (\ref{12}) in (\ref{13}), we get that, for each $z\in X$, $$\begin{array}{l}
     (\lambda^y_{x}\circ \lambda^y_{t})(z)=(xy)^{-*}* (x(ty)^{-*})*(xy)^{-*}* (xtz)=(xty)^{-*}*xy*(xy)^{-*}* (xtz) = \\
    \qquad  =(xty)^{-*}*(xtz)=\lambda^y_{xt}(z),
\end{array}$$ as desired.
\end{Proof}

Clearly, from Lemma \ref{Bachi} we get a semigroup morphism $$\Lambda=\prod_{y\in X}\lambda^y\colon (X,\cdot{})\to\prod_{y\in X}\End_{\Gp}(X, b_y).$$

\bigskip

From Lemma  \ref{Bachi} we also get the following result:

\begin{prop} Let $(X,[-,-,-],\cdot{})$ be a left near-truss. Define a ternary operation $\{-,-,-\}$ on $X$ setting $\{x,y,z\}=[y,xy,xz]$ for every $x,y,z\in X$. Then the algebra $(X,[-,-,-],\{-,-,-\})$ has the property that every $(X,b_y,m_y)$ is a left near-ring. Here $m_y$ is defined by $m_y(x,z)=\{x,y,z\}$  for every $x,y,z\in X$.\end{prop}


Since $\End_{\Gp}(X, b_y)\subseteq \End_{\Hp}(X, [-,-,-])$, from Lemma~\ref{Bachi} we also get that:

\begin{Lemma}\label{Bachi'} {\rm \cite[Proposition 3.5 and Remark 3.6]{Bre}} Let $(X,[-,-,-],\cdot{})$ be a left near-truss and $y$ be a fixed element of $X$. Then $\lambda^y\colon (X,\cdot{})\to\End_{\Hp}(X, [-,-,-])$, given by $\lambda^y\colon x\mapsto\lambda_x^y$, where $\lambda_x^y(z)=[y,xy,xz]$ for every $z\in X$, is a semigroup morphism.\end{Lemma}

\section{Generalizations of left near-trusses}

It is also convenient to define a generalization of left near-trusses: a {\em  left semi-near-truss} $(X,[-,-,-],\circ)$ is a set endowed with a ternary operation $[-,-,-]$ and a binary operation $\circ$, both associative  operations on $X$, such that left distributivity holds: $$x\circ[y,z,u]=[x\circ y,x\circ z,x\circ u].$$

\begin{exams} (1) 
Consider the triple $(\R, p, \cdot{})$, where $\R$ is the set of real numbers,\linebreak $p(x,y,z)=x$ for every $x,y,z\in\R$ and $\cdot$ denotes the usual operation of multiplication on $\R$. Then $(\R, p, \cdot{})$ is a left semi-near-truss.

(2) Fix any distributive lattice $(L,\vee,\wedge)$, let $p(x,y,z)=x\vee y\vee z$ be the semiheap operation on $L$ defined in Example~\ref{111}(d). Then $(L,p,\wedge)$ is a left semi-near-truss.\end{exams} 

A {\em left truss} $(X,[-,-,-],\circ)$ is a left near-truss for which the heap $(X,[-,-,-])$ is abelian. Similarly, a  {\em right truss} $(X,[-,-,-],\circ)$ is a right near-truss for which $(X,[-,-,-])$ is an abelian heap. A left truss that is also a right truss, is called a {\em truss}. Hence a truss $(X,[-,-,-],\circ)$ consists of an abelian heap $(X,[-,-,-])$, a semigroup $(X,\circ)$, and both distributivity laws hold.

\medskip

The main example of ring is, for any abelian group $(G,+)$, the endomorphism ring $(\End(G),+,\circ)$. Similarly, the main example of truss is, for any abelian heap $(X,[-,-,-])$, the endomorphism truss $(\End_\Hp(X),p,\circ)$ of $(X,[-,-,-])$. Here $\End_\Hp(X)$ denotes the set of all heap endomorphisms of $(X,[-,-,-])$.  The ternary operation $p$ on $\End_\Hp(X)$ is defined pointwise: for every $f,g,h\in \End_\Hp(X)$, that is, for every $f,g,h\colon X\to X$ that are heap endomorphisms of $X$, we have that $p(f,g,h)(x)=[f(x),g(x),h(x)]$ for every $x\in X$. 

\medskip

For any ring $R$, there are a canonical homomorphism $\mu\colon R\to \End(R,+)$ defined by $\mu(x)(y)=xy$ for every $x,y\in R$ and a canonical antihomomorphism $\rho\colon R\to\End(R,+)$ defined by $\rho(x)(y)=yx$ for every $x,y\in R$. Moreover, there is the compatibility $$\mu(x)(y)=\rho(y)(x)$$ for every $x,y$, and associativity of multiplication in the ring $R$ can be expressed by $\mu(x)\circ\rho(y)=\rho(y)\circ\mu(x)$  for every $x,y\in R$.

Similarly for trusses. Given any left near-truss $(H,[-,-,-],\cdot )$, we can define a mapping $\mu \colon (H,\cdot )\to (\End_{\Hp}(H,[-,-,-]),\circ )$, defined by $\mu(x)(y)=xy$ for every $x,y\in H$. Then $\mu$ is a semigroup morphism. Conversely, given any heap $(H,[-,-,-])$ and a further binary operation $\cdot$ on $H$, if the mapping $\mu\colon H\to \End_{\Hp}(H,[-,-,-])$, $\mu (x)(y)=xy$ for every $x,y\in H$, is a well defined semigroup morphism, then $(H,[-,-,-],\cdot )$ is a left near-truss. 

\medskip

Given any truss $(H,[-,-,-],\cdot{})$, we can define two mappings $\mu\colon H\to \End(H,[-,-,-])$, defined by $\mu(x)(y)=xy$ for every $x,y\in H$ and $\rho\colon H\to \End(H,[-,-,-])$ defined by $\rho(x)(y)=yx$ for every $x,y\in H$. Then $\mu$ is a truss morphism and $\rho$ is a truss antihomomorphism.

More precisely, having a left truss structure $(H,[-,-,-],\cdot{})$ is equivalent to having an abelian heap $(H,[-,-,-])$ with a truss morphism $$\mu\colon (H,[-,-,-],\cdot{})\to (\End_\Hp(H),[-,-,-],\circ).$$ 

A right truss structure $(H,[-,-,-],\cdot{})$ is equivalent to an abelian heap $(H,[-,-,-])$ with a truss antihomomorphism $\rho\colon (H,[-,-,-],\cdot{})\to (\End_\Hp(H),[-,-,-],\circ)$. 
A truss is a left truss that is also a right truss and there is compatibility $\mu(x)(y)=\rho(y)(x)$ for every $x,y$.

\section{Ideals and congruences on a left near-truss}

We now study congruences on left near-trusses. The following results  on left near-trusses are inspired by the corresponding results for trusses due to Brzezi{\'n}ski \cite{Bre}, though our terminology is partially different from his (we call ideals what he calls paragons).

A {\em congruence} on a left near-truss $(X, [-,-,-],\cdot{})$ is an equivalence relation $\sim$ on the set $X$ such that $[x,y,z]\sim[x',y',z']$ and $xy\sim x'y'$ for every $x,x',y,y',z,z'\in X$ such that $x\sim x'$, $y\sim y'$ and $z\sim z'$. 
Congruences on a left near-truss form a complete lattice.

\begin{Lemma}\label{7.1} Let $(X,[-,-,-],\cdot)$ be a left near-truss. For every normal subheap $S$ of the heap $(X,[-,-,-])$, let $\sim_S$ be the corresponding congruence on the heap $(X,[-,-,-])$, defined, for every $x,y\in X$, by $x\sim_Sy$ if there exists $s \in S$ such that
$[x, y,s] \in S$. The following conditions are equivalent:

{\rm (a)} $\sim_S$ is a congruence for the left near-truss $(X,[-,-,-],\cdot)$.

{\rm (b)} $[xp,xq,q]\in S$ and $[[p,q,x]y,xy,q]\in S$ for every $x,y\in X$ and every $p,q\in S$.\end{Lemma}

\begin{Proof} (a)${}\Rightarrow{}$(b) Suppose that $\sim_S$ is compatible with the multiplication $\cdot$ on the left near-truss $X$. Then $X/\sim_S$ is a left near-truss and there is a canonical homomorphism $\pi\colon X\to X/\sim_S$, which is a surjective left near-truss morphism. Fix $x,y\in X$ and $p,q\in S$. Then $\pi([xp,xq,q])=[\pi(x)\pi(p),\pi(x)\pi(q),\pi(q)]=\pi(q)=S$ because $\pi(p)=\pi(q)=S$, so $$[xp,xq,q]\in S.$$ Similarly $$\begin{array}{l}\pi([[p,q,x]y,xy,q])=[[\pi(p),\pi(q),\pi(x)]\pi(y),\pi(x)\pi(y),\pi(q)]= \\ \qquad =[\pi(x)\pi(y),\pi(x)\pi(y),\pi(q)]=\pi(q)=S,\end{array}$$ so $[[p,q,x]y,xy,q]\in S$.

(b)${}\Rightarrow{}$(a) Now assume that (b) holds for the normal subheap $S$. In order to prove that (a) holds, it suffices to show that if $x,y,z\in X$ and $y\sim_Sz$, then $xy\sim_Sxz$
    and $yx\sim_Szx$. From $y\sim_Sz$, we know that $[y,z,p]\in S$ for some $p\in S$. Then in the group $(X,b_p)$ we have that $[xy,xz,p]=xy-_pxz+_pp=xy-_pxz=xy-_pxz+_pxp-_pxp=[xy,xz,xp]-_pxp=[[xy,xz,xp],p,[p,xp,p]]=[x[y,z,p],xp,p]$. From $[y,z,p]\in S$ and the first property in (b) it follows that $[xy,xz,p]\in S$, so $xy\sim_Sxz$.

    Let us prove that $yx\sim_Szx$. We have that $y\sim_Sz$ implies $[y,z,q]\in S$ for some $q\in S$, so that, in the group $(X,b_q)$, $y-_qz+_qq\in S$, that is, the elements $y$ and $z$ of $(X,b_q)$ are congruent modulo the normal subgroup $S$ of $(X,b_q)$. Hence $y=p+_qz$ for some $p\in S$, i.e., $[p,q,z]=y$. It follows that $[p,q,z]x=yx$. From the second property in (b), we obtain that $S$ contains the element $[[p,q,z]x,zx,q]=[yx,zx,q]$, hence $yx\sim_Szx$, as we wanted to prove.
\end{Proof}

We will call  {\em ideal\/} in a left near-truss $(X, [-,-,-],\cdot{})$ any normal subheap $S$ of $(X, [-,-,-])$ such that $[xp,xq,q]\in S$ and $[[p,q,x]y,xy,q]\in S$ for every $x,y\in X$ and every $p,q\in S$.

As a consequence of Lemma \ref{7.1} we immediately get:

\begin{thm}\label{bijpp} Let $X$ be a left near-truss, $\Cal I(X)$ the set of all ideals of $X$, and $\Cal C(X)$ the set of all congruences of $X$. Then  there is a mapping $\Cal I(X)\to\Cal C(X)$, $S\mapsto\sim_S$, which  is a surjective mapping.\end{thm}

\bigskip

In view of Theorem~\ref{6.2} and Lemma~\ref{6.3}, it is convenient to study the structures $(X,+,\cdot{})$ for which $(X,+)$ is a group, not-necessarily abelian (so that probably we should be more careful and write also here $(X,+,-,0)$ as one does correctly in Universal Algebra), $(X,\cdot{})$ is a semigroup, and $w(x+z)=wx-(w\cdot 0)+wz$. Let's call them $J$-rings ($J$ for Jacobson), because our main example is, for any ring $(R,+,\cdot{})$, the $J$-ring $(R,+,\circ)$, where $\circ$ is the {\em Jacobson multiplication} $x\circ y=x+y+xy$.

Thus:

\begin{defin} A {\em $J$-ring} $(X,+,-,0,\cdot{})$ is a set $X$ with two binary operations $+$ and $\cdot$, a unary operation $-$ and a $0$-ary operation $0$ satisfying:

\noindent {\rm (i)} associativity of $+$;

\noindent {\rm (ii)} $x+0=0+x=x$ for every $x\in X$;

\noindent {\rm (iii)} 
$x+(-x)=(-x)+x=0$ for every  $x\in X$;

\noindent {\rm (iv)} associativity of $\cdot$;

\noindent {\rm (v)} ``left weak distributivity'' in the form $z(x+y)=zx-(z\cdot 0)+zy$  for every  $x,y,z\in X$.\end{defin}

The example of $J$-ring $(R,+,\circ)$ in Example~\ref{466} shows that in a $J$-ring one does not have in general $x\cdot 0=0$ nor $0\cdot x=x$.

\bigskip

We saw in Lemma~\ref{Bachi} that for any $J$-ring $(X,+,\cdot{})$ there is a semigroup morphism $\lambda\colon (X,\cdot{})\to\End_{\Gp}(X, +)$, given by $\lambda\colon x\mapsto\lambda_x$, where $\lambda_x(z)=-(x\cdot 0)+xz$ for every $x,z\in X$.

We have called {\em left weak distributivity} Property (v) in the definition of $J$-ring because it is a kind of modified distributivity. Passing to the new multiplication $\lambda$ corrects this alteration.

\bigskip

An {\em ideal} $I$ in a $J$-ring $(X,+,\cdot{})$ is a normal subgroup $N$ of the group $(X,+)$ such that $xn-x\cdot 0\in N$ and $(x+n)y-xy\in N$ for every $x,y\in X$ and every $n\in N$.

\begin{Lemma}\label{eq''} Let $(X,[-,-,-],\cdot)$ be a left near-truss and let $e$ be an element of $X$.
    Then there is a lattice isomorphism between the lattice of all ideals of the  $J$-ring $(X,b_e,\cdot)$ and the lattice of all congruences on $(X,[-,-,-],\cdot)$. This correspondence associates with every ideal $N$ of the  $J$-ring $(X,b_e,\cdot)$ the congruence $\sim_N$ on $(X,[-,-,-],\cdot)$ defined, for every $x,y\in X$, by $x\sim_N y$ if $x-y\in N$. Conversely, it associates to any congruence $\sim$ on $(X,[-,-,-],\cdot)$ the equivalence class $[e]_\sim$ of $e$ modulo $\sim$.
\end{Lemma}

\begin{Proof}Let $(X,[-,-,-],\cdot)$ be a left near-truss and let $e$ be an element of $X$. By Corollary~\ref{3.4}, normal subgroups of the group $(X,b_e)$ are exactly the normal subheaps of the heap $(X,[-,-,-])$ that contain $e$. By Proposition~\ref{eq}, there is a bijection $N\mapsto\sim_N$ between the set of all 
normal subheaps of $(X,[-,-,-])$ that contain $e$ and the set of all congruences on the heap $(X,[-,-,-])$. Hence we have a lattice isomorphism between the lattice of all congruences on the heap $(X,[-,-,-])$ and the lattice of normal subgroups of the group $(X,b_e)$. Therefore it suffices to prove that, in this lattice isomorphism, ideals of the  $J$-ring $(X,b_e,\cdot)$ correspond to equivalence relations compatible with the multiplication $\cdot$.

   Let $N$ be an ideal of the  $J$-ring $(X,b_e,\cdot)$. Since $N$ is a normal subgroup of $(X,b_e)=(X,+)$, the relation $\sim_N$ is an equivalence relation compatible with the group operation $b_e$, hence with the ternary operation $[-,-,-]$ on $X$ (Proposition~\ref{eq}). Let us prove that $\sim_N$ is compatible with the multiplication $\cdot$ on $X$. Clearly, it suffices to show that if $x,x',y\in X$ and $x\sim_N x'$, then $xy\sim_N x'y$ and $yx\sim_N yx'$. Now $x\sim_N x'$ implies that $x'=x+n$ for some $n\in N$, so that $xy-x'y=xy+(-((x+n)y))=-((x+n)y-xy)\in N$, that is, $xy\sim_N x'y$. In order to show that $yx\sim_N yx'$, we must prove that $yx-yx'\in N$. Now  $yx-yx'=yx-y(x+n)=yx-(yx-y0+yn)=yx-yn+y0-yx=yx-y0+y0-yn+y0-yx=(yx-y0)-(yn-y0)-(yx-y0)$. But $yn-y0$ belongs to $N$, because $N$ is an ideal, so that its opposite $-(yn-y0)$ belongs to $N$, hence its conjugate $(yx-y0)-(yn-y0)-(yx-y0)\in N$ because $N$ is a normal subgroup.
   
   Conversely, let $\sim$ be a congruence on the left near-truss $(X,[-,-,-],\cdot)$, so that $\sim$ is also a congruence on the $J$-ring $(X,[-,-,-],\cdot)$. Let $\pi\colon X\to X/\sim$ be the canonical projection, and let $N:=[e]_\sim$ denote the equivalence class of $e$ modulo $\sim$. Then, for every $x,y\in X$ and every $n\in N$, we have $\pi(xn-x\cdot 0)=\pi(x)\pi(n)-\pi(x)\cdot \pi(e)=\pi(x)N-\pi(x)\cdot N=N$, so $xn-x\cdot 0\in N$. Similarly, $\pi((x+n)y-xy)=(\pi(x)+N)\pi(y)-\pi(x)\pi(y)=N$.
\end{Proof}

Let us go back to the morphism $\lambda$, recalling that, for any $J$-ring $(X,+,\cdot{})$, the semigroup morphism $\lambda\colon (X,\cdot{})\to\End_{\Gp}(X, +)$ is defined by $\lambda\colon x\mapsto\lambda_x$, where $\lambda_x(z)=-(x\cdot 0)+xz$ for every $x,z\in X$. Notice that the condition $xn-x\cdot 0\in N$ for every $x\in X$, $n\in N$, in the definition of ideal is equivalent to $\lambda_x(N)\subseteq N$, and has as a consequence that $\lambda$ induces a semigroup morphism $\lambda'\colon (X,\cdot{})\to\End_{\Gp}(X/N, +)$. The condition $(x+n)y-xy\in N$ for every $x,y\in X$ and every $n\in N$ is equivalent to the fact that $\lambda_x=\lambda_y$ for every $x,y\in X$ with $x\sim_Ny$, so that $\lambda'$ induces a semigroup morphism $(X/N,\cdot{})\to\End_{\Gp}(X/N, +)$.

\section{Huq commutator and Smith commutator for left near-trusses}\label{C}

The question whether ``Huq${}={}$Smith'' or ``Huq${}\ne{}$Smith'' typically concernes varieties of algebras that are semiabelian categories. Huq commutator is a category-theoretic concept introduced
by Huq \cite{14}. The Smith commutator was introduced by Smith \cite{20} for varieties of algebras that are Mal’tsev varieties. The two notions of Huq commutator and Smith commutator coincide, for instance, in the varieties of
groups, Lie algebras, associative algebras, and non-unital rings. On the contrary, Huq$\ne$Smith for digroups, loops, and near-rings. 

MacLane \cite{8diHuq14} was the first to observe that for a group $G$ to be abelian it is both necessary and sufficient that the multiplication $\cdot\colon G\times G\to G$ is a homomorphism of the direct product $G\times G$ into $ G$.
In any semiabelian category the full subcategory of abelian objects is abelian.
In the semiabelian category of non-unital rings the abelian objects are the rings with zero multiplication.
But the categories we are studying in this paper, i.e., the category of heaps and that of (near-)trusses, are very far from being semiabelian. 
In the categories of heaps and trusses there is not a null object, and in the category of trusses there are not objects with zero multiplication (but there are objects with constant multiplication, where $x\cdot y=a$ for every $x,y\in X$ and $a$ is a fixed object of $X$).

Now the Huq commutator concerns the construction of the commutator of two morphisms with the same codomain: it concerns, for any pair $f\colon A\to C$ and $g\colon B\to C$ of morphisms, the construction of a morphism which universally makes them commute. Typically, in a semiabelian category, the two morphisms $f$ and $g$ are the inclusions of two normal subobjects of an object $X$ of the semiabelian category into $X$. Now, in a semiabelian category, the normal subobjects of an object $A$ are in one-to-one correspondence with the congruences (=kernel pairs) on $A$ \cite[Theorem~3.4]{Borceux}. We saw in Theorem~\ref{bijp} and we will see in Theorem~\ref{bijpp} that this is not the case for heaps and near-trusses, and there is not even a null object in these categories. Hence the question 
 ``Huq=Smith?'' must be revisited for the algebraic structures with a ternary operation we are studying.

If $\Cal C$ is a category with finite limits, for any congruence $R$ on an object $X$, let $d_0\colon R\to X$ and $d_1\colon R\to X$ denote the first and the second projections of $R$. For any two congruences $R$ and $S$, let $R\times_X S$ denote the pullback $$\xymatrix{
R\times_X S\ar[r]^{p_1}\ar[d]_{p_0}&S\ar[d]^{d_0} \\
R\ar[r]_{d_1} &\phantom{.}X. }$$ For instance, if $X$ is a heap (or a left near-truss), and $R$ and $S$ are congruences on $X$, then $$R\times_X S=\{\,(x,y,z)\mid x,y,z\in X,\ xRy\ \mbox{\rm and}\ ySz\,\}.$$ The canonical {\em connector} between $R$ and $S$ (\cite[Example 1.2]{BG}, \cite{Ped1} and \cite{Ped2}) is the mapping $p\colon R\times_XS\to X$ defined by $p(x,y,z)=[x,y,z]$ for every $(x,y,z)\in R\times_XS$. The Smith commutator of $R$ and $S$ is the smallest congruence $T$ on $X$ for which the mapping \begin{equation}
    R\times_X S\to X/T,\qquad (x,y,z)\mapsto [p(x,y,z)]_T,\label{sru}
\end{equation} is a heap (left near-truss, resp.) morphism.

Fix an element $e\in X$, and consider the subheap $$\{\,(x,e,z)\mid x,z\in X,\ xRe\ \mbox{\rm and}\ eSz\,\}$$ of $X^3$. That is, $$\{\,(x,e,z)\mid x\in[e]_R\ \mbox{\rm and}\ z\in[e]_S\,\},$$  which is clearly in a one-to-one correspondence with the cartesian product $A\times B$, where $A:=[e]_R$ and $B:=[e]_S$ are the equivalence classes of $e$ modulo $R$ and $S$ respectively, that is, the normal subgroups (ideals) corresponding to $R$ and $S$ in $(X,b_e)$ (Proposition~\ref{eq} and Lemma~\ref{eq''}). Thus the mapping (\ref{sru}) restricts to a well defined group morphism ($J$-ring morphism) $$A\times B\to X/T,\qquad (x,z)\mapsto [p(x,e,z)]_T=[x+_ez]_T.$$ Thus $[e]_T\supseteq C_e$, where $C_e$ is the smallest ideal of the $J$-ring $(X,b_e,\cdot)$ for which the mapping $$A\times B\to X/C_e,\qquad (x,z)\mapsto p(x,e,z)+C_e=x+_ez+_eC_e$$ is a well defined group morphism ($J$-ring morphism). 
Now, in our case with a ternary operation, the question ``Huq=Smith''  asks whether the congruence $T$ and the ideal $C_e$ correspond to each other for every possible choice of $e$, that is, whether $T{}={}\sim_{C_e}$ or, equivalently, $C_e=[e]_T$.
Thus, in Theorem~\ref{4.1} we have proved exactly that Huq=Smith holds for heaps.

\bigskip

Let us pass to consider the case of left near-trusses. Since Huq${}\ne{}$Smith for near-rings \cite{J}, and every left near-ring $(N,+,\cdot{})$ naturally produces a left near-truss $(N,[-,-,-],\cdot{})$ (Example~\ref{6.2*}(2)),
it is not surprising that Huq${}\ne{}$Smith for left near-trusses. The example is the same as the example in \cite[Section~4]{J}, except for the fact that in the passage from near-rings to near-trusses we must pass from the {\em right} near-ring of \cite[Section~4]{J} to {\em left} near-trusses as we have done in most of this paper. The example that Huq${}\ne{}$Smith for left near-trusses is the following. Let $M$ be any abelian group with a nonzero proper subgroup $K$, and let $X$ be the direct product $X:=M^3=M\times M\times M$ of three copies of the group $M$ with the usual ternary operation $[x,y,z]:=x-y+z$ for every $x,y,z\in X=M^3$. As a multiplication in $X$, define $$(n_1,n_2,n_3)(m_1,m_2,m_3):=\left\{\begin{array}{ll}
    (m_2,0,0)&  \mbox{\rm if\ $n_2\ne 0$ and $n_3\ne 0$}\\ (0,0,0)&  \mbox{\rm  if\ $n_2= 0$ or $n_3= 0$.}\end{array}\right.$$ It is possible to see that $X$ is a left near-ring and that the Huq commutator and the Smith commutator for the two ideals $A = M\times K\times\{0\}$ and $B = M\times\{0\}\times M$ of the left near-ring $(X,+,\cdot)$ are different \cite[Section~4]{J}. Let us be more precise to explain in detail what occurs for the corresponding left near-truss. 

    Given the left near-ring $(X=M^3,+,\cdot)$ defined in the previous paragraph, then $X$ yields a left near-truss $(X,[-,-,-],\cdot{})$  according to Example~\ref{6.2*}(2). Consider the two ideals $A = M\times K\times\{0\}$ and $B = M\times\{0\}\times M$ of the left near-ring $X$. Correspondingly, there are two congruences $R$ and $S$ on the left near-truss $(X,[-,-,-],\cdot{})$ for which $A=[0]_R$ and $B=[0]_S$. Here $0$ is the identity of the group $(X=M^3,+)$. In our terminology one has that $R\times_XS=\{\,(a+x,x,x+b)\mid a\in A,\ x\in X,\ b\in B\,\}$. There is a one-to-one correspondence between the ideals $C$ of the left near-ring $(X=M^3,+,\cdot)$ and the congruences $T$ on the left near-truss $(X,[-,-,-],\cdot{})$. The mapping $p\colon R\times_XS\to X/C$, $p(a+x,x,x+b)=a+x+b+C$, is a left near-ring morphism if and only if it is a left near-truss morphism $p\colon R\times_XS\to X/T$. Now $p$ is a heap morphism because $X=M^3$ is an abelian heap. Hence $p\colon R\times_XS\to X/T$ is a left near-truss morphism if and only if $p\colon R\times_XS\to X/C$ respects multiplication, that is, if and only if $p((a_1+x_1,x_1,x_1+b_1)(a_2+x_2,x_2,x_2+b_2))=p(a_1+x_1,x_1,x_1+b_1)p(a_2+x_2,x_2,x_2+b_2)$ for every $a_1,a_2\in A$, $x_1,x_2\in X$, $b_1,b_2\in B$. Equivalently, if and only if $(a_1+x_1)(a_2+x_2)-x_1x_2+(x_1+b_1)(x_2+b_2))+C=(a_1+x_1+b_1)(a_2+x_2+b_2)+C$ for every $a_1,a_2\in A$, $x_1,x_2\in X$, $b_1,b_2\in B$. It follows that the ideal $C$ of the left near-ring $X$ corresponding to the Smith commutator $[R,S]$ in the left near-truss $X$ is generated by the set $S=\{\,(a_1+x_1)(a_2+x_2)-x_1x_2+(x_1+b_1)(x_2+b_2))-(a_1+x_1+b_1)(a_2+x_2+b_2)\mid a_1,a_2\in A, \ x_1,x_2\in X,\ b_1,b_2\in B\,\}.$ The same computation, specialized to the case $x_1=x_2=0$, shows that the ideal $C'$ of the left near-ring $X$ corresponding to the the Huq commutator of $R$ and $S$ is the ideal of the left near-ring $X$ generated by the set $H=\{\,a_1a_2+b_1b_2-(a_1+b_1)(a_2+b_2)\mid a_1,a_2\in A, \ b_1,b_2\in B\,\}.$ 

    In the example of \cite[Section~4]{J} cited above, the Huq commutator of $R$ and $S$ is contained in $K\times\{0\}\times\{0\}$ because in this case $H=\{\,(m_1,k_1,0)(m_2,k_2,0)+(m_3,0,m_4)(m_5,0,m_6)+(m_1+m_3,k_1,m_4)(m_2+m_5,k_2,m_6)\mid  m_1,m_2,m_3,m_4,m_5,m_6\in M, \ k_1,k_2\in K\,\}$, and one has that $(m_1,k_1,0)(m_2,k_2,0)=0$, $(m_3,0,m_4)(m_5,0,m_6)=0$, and $(m_1+m_3,k_1,m_4)(m_2+m_5,k_2,m_6)\in K\times\{0\}\times\{0\}$. 

    On the contrary, one sees that the set $S$ of generators contains (for $k$ a non-zero element of $K$, $m$ a non-zero element of $M$, $t$ and element of $M$, $a_1=(0,-k,0)$, $a_2=(0,0,0)$, $x_1=(0,k,m)$, $x_2=(0,t,0)$, $b_1=(0,0,-m)$ and $b_2=(0,0,0)$) the element $(a_1+x_1)(a_2+x_2)-x_1x_2+(x_1+b_1)(x_2+b_2))-(a_1+x_1+b_1)(a_2+x_2+b_2)=(0,0,m)(0,t,0)-(0,k,m)(0,t,0)+(0,k,0)(0,t,0)-(0,0,0)=(t,0,0)$. Hence the Smith commutator contains $M\times\{0\}\times\{0\}$, and therefore the Smith commutator is not contained in the Huq commutator in this example.

\section{Semidirect product of left near-trusses}

Let us pass to consider semidirect product of left near-trusses, equivalently idempotent endomorphisms of left near-trusses. 

\begin{prop}\label{boh'} Let $X\ne\emptyset$ be a left near-truss, $Y$ be a subnear-truss of $X$, and $\omega$ a congruence on the left near-truss $X$. The following conditions are equivalent:

{\rm (a)} $Y$ is a set of representatives of the equivalence classes of $X$ modulo $\omega$ (i.e., $Y\cap[x]_\omega$ is a singleton for every $x\in X$).

{\rm (b)} There exists an idempotent left near-truss endomorphism of $X$ whose image is $Y$ and whose kernel is $\omega$.

{\rm (c)} For every $e\in Y$, there exists an idempotent $J$-ring endomorphism of the $J$-ring $(X,b_e,\cdot)$ whose image is the sub-$J$-ring $Y$ of $(X,b_e,\cdot)$ and whose kernel is the ideal $[e]_\omega$ of $(X,b_e,\cdot)$. 

{\rm (d)} There exist an element $e\in Y$ and an idempotent $J$-ring endomorphism of the $J$-ring $(X,b_e,\cdot)$ whose image is the sub-$J$-ring $Y$ of $(X,b_e,\cdot)$ and whose kernel is the ideal $[e]_\omega$ of $(X,b_e,\cdot)$. 

{\rm (e)} For every $e\in Y$, the group $(X,b_e)$ is the semidirect product of its subgroup $Y$ and its normal subgroup $[e]_\omega$. 

{\rm (f)} There exist an element $e\in Y$ such that the group $(X,b_e)$ is the semidirect product of its subgroup $Y$ and its normal subgroup $[e]_\omega$. 

{\rm (g)}
For every $a\in X$ and every $c\in Y$ there exist a unique element $b\in X$ and a unique element $d\in Y$ such that $a=p(b,c,d)$ and $b\,\omega\, c$.

  {\rm (h)}
For every $a\in X$ and every $c\in Y$ there exist a unique element $b\in X$ and a unique element $d\in Y$ such that $a=p(d,c,b)$ and $b\,\omega\, c$.

{\rm (i)} The mapping $g\colon Y\to X/\omega$, defined by $g(y)=[y]_\omega$ for every $y\in Y$, is a left near-truss isomorphism.

{\rm (l)} There exists a left near-truss endomorphism of $X$ whose kernel is $\omega$ and whose restriction to $Y$ is the inclusion of $Y$ in $X$.
\end{prop} 

\begin{Proof} We will constantly make use of Proposition~\ref{boh} and its proof.

    (a)${}\Rightarrow{}$(b) We must prove that the heap endomorphism $f$ defined in the proof of (a)${}\Rightarrow{}$(b) of Proposition~\ref{boh} also respect multiplication. The mapping $f\colon X\to X$ associates with every $x\in X$ the unique element of the singleton $Y\cap[x]_\omega$. We must prove that if $x,y\in X$, then $f(xy)=f(x)f(y)$. Let $x',y'$ be the unique elements of  $Y\cap[x]_\omega$ and $Y\cap[y]_\omega$,  respectively, so that $f(x)=x'$ and $f(y)=y'$. Then $x\,\omega\, x'$ and $y\,\omega\, y'$ because $f$ is an idempotent mapping and $\omega$ is its kernel. Since $\omega$ is a congruence for a left near-truss, we get that $xy\,\omega\, x'y'$. Thus $x'y'$ belongs to the congruence class of $xy$ modulo $\omega$. Also, $x'y'$ belongs to $Y$, since $Y$ is multiplicatively closed. Thus $x'y'\in Y\cap[xy]_\omega$, so $f(xy)=x'y'=f(x)f(y)$.

All the other implications (b)${}\Rightarrow{}$(c),
(c)${}\Rightarrow{}$(d), 
(d)${}\Rightarrow{}$(b), 
(b)${}\Rightarrow{}$(a), 
(c)${}\Leftrightarrow{}$(e), 
(d)${}\Leftrightarrow{}$(f), 
(e)${}\Leftrightarrow{}$(g), 
(e)${}\Leftrightarrow{}$(h), (a)${}\Leftrightarrow{}$(i), (i)${}\Rightarrow{}$(l) and 
(l)${}\Rightarrow{}$(i)
follow immediately from the corresponding implications in the proof of Proposition~\ref{boh}. \end{Proof}

If $X$ is a left near-truss and its subnear-truss $Y$ and the congruence $\omega$ on $X$ satisfy the equivalent conditions of Proposition~\ref{boh'}, the left near-truss $X$ is the {\em (inner) semidirect product} of $Y$ and $\omega$, and we write $X=\omega\rtimes Y$. 

\begin{prop}
    Let $X$ be a left near-truss. Then there is a one-to-one correspondence between:

    {\rm (a)} The set $A$ of all idempotent left near-truss endomorphisms of $X$.

     {\rm (b)} The set $B$ of all pairs $(\omega,Y)$, where $\omega$ is a congruence on the left near-truss $X$, $Y$ is a subnear-truss of $X$, and $X=\omega\rtimes Y$.
\end{prop}

\section{Derivations of trusses}

In \cite[Definition 3.10]{Brzez} an interesting notion of derivation in a truss is determined. Let $T$ be a truss. A heap homomorphism $D\colon T\to T$ is called a {\em derivation} if, for all $a,b\in T$,
$D(ab) = [D(a)b, ab, aD(b)]$. 

Let us see that elementary properties of derivations for rings also have a suitable analogue for derivations of trusses:

(1) Derivations on $T$ form an abelian heap which we denote by $\Der(T).$

\begin{Proof}: We will show that $\Der(T)$ is a subheap of $\End(T,[-,-,-])$, that is, that if $D_1,D_2,D_3\in \Der(T)$, then $[D_1,D_2,D_3]$ belongs to $\Der(T)$. In fact 

\centerline{ $\begin{array}{l}[D_1,D_2,D_3](ab)=[D_1(ab),D_2(ab),D_3(ab)]= \\ \qquad=[[D_1(a)b, ab, aD_1(b)],[D_2(a)b, ab, aD_2(b)],[D_3(a)b, ab, aD_3(b)]]= \\ \qquad=[[D_1(a)b, D_2(a)b, D_3(a)b],[ab,ab,ab],[aD_1(b),aD_2(b),aD_3(b)]= \\ \qquad=[[D_1,D_2,D_3](a)b, ab, a[D_1,D_2,D_3](b)]).\end{array}$}
\quad \quad \quad \quad \quad \quad \quad \quad \quad \quad \quad \quad \quad \quad \quad \quad \quad \quad \quad \quad \quad \quad \quad \quad \quad \quad \quad \quad \quad \quad \quad \quad \quad \quad \quad \quad \quad \quad \quad \end{Proof}

Already in this proof we use right and left distributivity and the fact that the heap is abelian. This explains why in this Section we deal with trusses and not with left near-trusses like in the rest of the paper.

(2) For derivations of rings we have that if $D,D'\colon R\to R$ are derivations of a ring $R$, then $DD'-D'D$ is a derivation of $R$. The analogous property for derivations of trusses is the following:

\begin{Lemma} Let $D,D'\colon T\to T$ be two derivations of a truss $(T,[-,-,-],\cdot{})$. Let $\iota_T\colon T\to T$ be the identity mapping. Then $[DD',D'D,\iota_T]\colon T\to T$ is a derivation of the truss $T$.
\end{Lemma}

\begin{Proof} Clearly, $[DD',D'D,\iota_T]\colon T\to T$ is a heap morphism of the heap $(T,[-,-,-])$. In order to show that it is a truss derivation, we must show that, for all $a,b\in T$, \begin{equation}[DD',D'D,\iota_T](ab)=[[DD',D'D,\iota_T](a)\cdot b,ab,a\cdot[DD',D'D,\iota_T](b)].\label{der}
\end{equation} Making use of Remark~\ref{2.2}, we see that the term on the left hand side of (\ref{der}) is \begin{equation}
\noindent\begin{array}{l}[DD',D'D,\iota_T](ab)=[DD'(ab),D'D(ab),ab]=[D([D'(a)b, ab, aD'(b)]),\dots,ab]= \\ \qquad =[[D(D'(a)b),D(ab),D(aD'(b)],\dots,ab]= \\ \qquad = [[[(DD'(a))b,D'(a)b,D'(a)D(b)],[D(a)b,ab,aD(b)],[D(a)D'(b),aD'(b), \\ \qquad\qquad aDD'(b)],\dots,ab]= \\ \qquad = 
[(DD'(a))b,D'(a)b,D'(a)D(b),D(a)b,ab,aD(b),D(a)D'(b),aD'(b),aDD'(b), \\ \qquad\qquad
    (D'D(a))b,D(a)b,D(a)D'(b),D'(a)b,ab,aD'(b),D'(a)D(b),aD(b),aD'D(b),
    ab]= \\ \qquad =
[(DD'(a))b,D'(a)b,D'(a)b,D'(a)D(b),D'(a)D(b),D(a)b,D(a)b,ab,ab,aD(b),aD(b), \\ \qquad\qquad D(a)D'(b),D(a)D'(b),aD'(b),aD'(b),aDD'(b),
    (D'D(a))b,aD'D(b),
    ab]= \\ \qquad =
[(DD'(a))b,aDD'(b),
    (D'D(a))b,aD'D(b),
    ab]= \\ \qquad =
    [(DD'(a))b,(D'D(a))b,aDD'(b),
    aD'D(b),
    ab].\end{array}\end{equation}
Similarly, the term on the right hand side of (\ref{der}) is \begin{equation}\begin{array}{l}[[DD',D'D,\iota_T](a)\cdot b,ab,a\cdot[DD',D'D,\iota_T](b)]= \\ \qquad=
[[DD'(a),D'D(a),a]\cdot b,ab,a\cdot[DD'(b),D'D(b),b]= \\ \qquad =
[DD'(a)\cdot b,D'D(a)\cdot b,ab,ab,a\cdot DD'(b),a\cdot D'D(b),ab]= \\ \qquad = 
[DD'(a)\cdot b,D'D(a)\cdot b,a\cdot DD'(b),a\cdot D'D(b),ab].\end{array}\end{equation} This concludes the proof of the Lemma.\end{Proof}

(3) One of the first examples of derivation for a ring $R$ is, for any fixed element $a\in R$, the mapping $d_a\colon R\to R$ defined by $d_a(x)=ax-xa$ for every $x\in R$. The analogue for trusses is the following:

\begin{thm} Let $(T,[-,-,-],\cdot{})$ be a truss and $a$ an element of $T$. Let $\lambda_a\colon T\to T$ and $\rho_a\colon T\to T$ be left (right) multiplication by $a$ respectively. Then the mapping $D_a:=[\lambda_a,\rho_a,\iota_T]\colon T\to T$ is a derivation of $T$.
    \end{thm}

{\sc Proof.}\ We must prove that, for every $x,y\in T$, $D_a(xy)=[D_a(x)y,xy,xD_a(y)]$. Now \begin{equation*}\begin{array}{l}[D_a(x)y,xy,xD_a(y)]=[[ax,xa,x]y,xy,x[ay,ya,y]]= \\ \qquad =
    [axy,xay,xy,xy,xay,xya,xy]=[axy,xya,xy]=D_a(xy).\qquad\rule{1ex}{1ex}\end{array}\end{equation*}

\end{document}